\documentclass[twoside]{IEEEtran} 
\usepackage[numbers]{natbib}
\usepackage{amsfonts,amsmath,amsthm,amssymb}
\usepackage{url}
\usepackage{dsfont}
\usepackage{graphicx} 
\usepackage[utf8]{inputenc}
\usepackage{subcaption}

\newtheorem{lemme}{Lemma}
\newtheorem{prop}{Proposition}
\newtheorem{theorem}{Theorem}

\newtheorem{definition}{Definition}
\newtheorem{rque}{Remark}
\newtheorem{ass}{Assumption}

\newcommand{\Pvec}{\mathbf{P}}
\newcommand{\Svec}{\mathbf{S}}
\newcommand{\pen}{\operatorname{pen}}
\newcommand{\AnT}{\mathbf{A}_{n \times n}^{1:T}}
\newcommand{\anT}{\mathbf{a}_{n \times n}^{1:T}}
\newcommand{\ZnT}{\mathbf{Z}_{n}^{1:T}}
\newcommand{\znT}{\mathbf{z}_{n}^{1:T}}
\newcommand{\dyn}{\text{dyn}}
\newcommand{\ML}{\text{ML}}

\usepackage{xcolor}
\newcommand{\blue}[1]{\textcolor{black}{#1}}

\begin{document}
\title{
Consistent model selection in a collection of stochastic block models
}

\author{Lucie~ARTS%
\thanks{Lucie Arts is with the Laboratoire de Probabilités, Statistique et Modélisation, Sorbonne Université and Université Paris Cité, CNRS,  F-75005 Paris, France (e-mail: arts@lpsm.paris)}}

\maketitle

\begin{abstract}
We introduce the penalized Krichevsky-Trofimov (KT) estimator as a convergent method for estimating the number of nodes clusters when observing multiple networks within both multi-layer and dynamic Stochastic Block Models. We establish the consistency of the KT estimator, showing that it converges to the correct number of clusters in both types of models when the number of nodes in the networks increases. Our estimator does not require a known upper bound on this number to be consistent. Furthermore, we show that these consistency results hold in both dense and sparse regimes, making the penalized KT estimator robust across various network configurations. We illustrate its performance on synthetic datasets.
\end{abstract}


\section{Introduction}

Network analysis has become a widely used tool across numerous scientific domains, providing a powerful framework for analyzing complex systems. From social networks  \citep{use_social_network} to transportation systems \citep{use_transport}, relationships between entities are naturally represented as networks, where nodes denote entities and edges represent their interactions.  

Clustering, a subfield of network analysis, is especially focused on uncovering hidden structures such as communities or clusters, within the network. These structures can provide insights into the underlying processes that generate the observed data and clustering methods aim to group nodes that behave similarly. A key challenge in this context is order estimation, which refers to determining the number of communities or clusters in the network. Model selection plays a critical role here, as an inaccurate choice of the number of clusters can lead to misleading interpretations of the network’s structure.

Various statistical methods have been developed to address the model selection issue, ranging from information theoretic criteria like the Akaike Information Criterion~\citep[AIC,][]{AIC} and the Bayesian Information Criterion \citep[BIC,][]{BIC_vrai_modele} to more specialized approaches like the Integrated Classification Likelihood  \citep[ICL,][]{intro_ICL,ICL_Latouche,ICL_daudin} that is suited to the presence of latent variables. 
The Krichevsky-Trofimov (KT) estimator, a Bayesian estimator rooted in information theory, has been initially developed for estimating the parameters of a categorical distribution without prior knowledge \citep{KT}.  
Its penalized version is linked to minimum description length encoding. It has been used for model selection and appears to be consistent in various models: in HMM  \citep{Liu_Narayan,gassiat_boucheron}, in context tree \citep{csiszar}, in Markov chains with Markov regime \citep{chambaz_matias} and in stochastic block models \citep[SBM,][]{CL20}. Recently, the Singular Bayesian Information Criterion \citep{BIC_singular_modele} has emerged as a promising path for understanding model selection, particularly in complex or high-dimensional data settings. However, this criterion remains theoretical. Proofs of consistency for order estimators frequently rely on an a priori bound on the order, simplifying the analysis but introducing a restriction that may not always be practical nor necessary. Establishing consistency without such bounds is a significant challenge, as it requires addressing the complexities of unbounded parameter spaces while ensuring the robustness of the estimator. Another important aspect of model selection is the pursuit of minimal penalties. In the literature on order estimation, there is a focus on minimizing penalties, as this leads to a lower probability of underestimating the model. For example, \cite{van_Handel} underlines the importance of minimizing penalties to reduce the risk of missing essential clusters or model components. This highlights the complexity of model selection, where the trade-off between non-underestimation and non-overestimation is crucial for obtaining reliable and interpretable results. 

Stochastic Block Models (SBMs) have emerged as a prominent class of statistical random graph models that facilitate node clustering. Originally introduced by \cite{first_SBM}, SBMs incorporate latent variables at the nodes, which assume values from a finite set. These latent variables represent node groups, and the interactions between nodes are governed by the groups to which they belong \citep{bickel2009nonparametric,SBM_commu}. This framework allows SBMs to capture key aspects of real networks, shedding light on the processes that drive their formation and evolution. In the context of the Stochastic Block Model, the identifiability and consistency of cluster estimation depend on whether the number of clusters $k$ is known or needs to be estimated. When $k$ is fixed and known, methods like spectral clustering, which leverages the eigenvalue decomposition of adjacency matrices, are widely used and consistent under conditions such as a sufficiently large spectral gap and well-defined block structures \citep{spectral_consistency, spectral_clustering}. However, when $k$ is unknown, additional challenges arise, as both the model order and the cluster structure must be estimated simultaneously. In this case, criteria such as the Integrated Classification Likelihood (ICL) \citep{ICL_daudin} or penalized likelihood approaches are employed.

Despite the variety of approaches, there remains a lack of consistent criteria for reliably determining the order of a SBM, especially in networks that exhibit complex and heterogeneous patterns of connectivity. Recently, a significant breakthrough was achieved in \cite{CL20}, who established a consistency result on the number of clusters for the SBM that does not require known upper-bound on the number of clusters. This result, which applies to both sparse (the number of edges grows more slowly than the number of nodes, leading to fewer connections) and dense (the number of edges increases in proportion to the number of nodes, resulting in higher connectivity) regimes, uses the penalized Krichevsky-Trofimov estimator. This represents a major advancement in the field.
 
Previous consistency results for model selection and order estimation have been developed in the context of a single graph. The aim of this article is to extend this result to a collection of graphs on the same set of nodes, introducing a common framework that encompasses two key models: the multi-layer SBM (MLSBM) and the dynamic SBM (DynSBM). In the MLSBM \citep{def_multi_SBM}, multiple graphs with the same latent structure are observed simultaneously, making it particularly relevant for networks that evolve across multiple dimensions or layers. On the other hand, in the DynSBM \citep{def_Dyn_SBM_MM} a sequence of graphs is observed and node groups evolve over time via independent  Markov chains. While in MLSBM, index $T$ corresponds to the level (i.e. number of layers), it represents time (or any unidimensional gradient along which the network are organized) in the DynSBM. This article investigates an asymptotic framework where the number of nodes in each graph increases (i.e. $n\rightarrow \infty$) while the number of networks $T$ remains fixed, a setting that, to the best of our knowledge, has not been extensively studied in previous work for a collection of SBMs. 

In this paper, the penalized KT estimator is studied in two regimes: the dense regime, where the probability of having an edge between nodes is considered to be constant, and the sparse regime, where this probability decreases to zero with $n$, having order $\rho_n$. The dense regime, although theoretically interesting, is not considered realistic, as many real-world networks exhibit sparsity. In contrast, the sparse regime is more relevant for such networks, as the edge probability tends to be small, and one challenge lies in controlling the amount of information needed to accurately estimate the clusters. In this sparse regime, we assume that the expected degree of a given node grows to infinity, that is, $n \rho_n \rightarrow \infty$. However, as the sparsity increases further, it becomes more difficult to detect the underlying community structure. At some point, the information available is insufficient to reliably estimate the number of clusters or even detect their presence. In fact, $n\rho_n= \Omega(\log n)$ is the detection threshold for complete recovery of the communities in SBMs \citep{threshold,Mossel},  but we only need $n\rho_n \rightarrow \infty$ at any rate to weakly recover the community labels. In this paper, we prove the consistency of the estimator up to this detection threshold in MLSBM and DynSBM.
 
In summary, this article aims to advance network analysis methodologies by introducing new consistency results for model selection in both multi-layer and dynamic Stochastic Block Models. We begin in Section \ref{def_not} by defining the models and notation relevant to our study. Section \ref{KT_est} introduces the penalized KT estimator and presents the main consistency theorem. Section \ref{proof_thm} outlines the key ideas behind the proof of this theorem with more detailed proofs provided in the Appendix. Finally, Section \ref{simu} includes simulations to demonstrate the practical effectiveness of the estimator.

\section{Definition of a collection of SBMs} \label{def_not}

We start by introducing the definitions of a multi-layer SBM and a dynamic SBM, two models for collections of networks. To simplify the notation, our models are restricted to undirected random graphs without self-loops. The generalization can be achieved with minimal additional modifications.

Consider a collection of binary undirected graphs with no self-loops formed by $n$ nodes and indexed by some $t=1,\dots,T$ (whether level or time). We observe an adjacency array $\mathbf{A}_{n\times n}^{1:T}=\left(A_{n\times n}^{1},\dots, A_{n \times n}^{T}\right)$, where for any $t$ in $ \{1,\dots,T\},$ each matrix $A_{n \times n}^{t} \in \{0,1\}^{n \times n}$ is symmetric and has diagonal entries equal to zero. In both models, we assume that, for each $t$, the $n$ nodes are split into $k$ latent groups, as encoded by the random variables $\ZnT= \left( Z_i^t \right)_{1\leq t \leq T, 1 \leq i \leq n}$ with $Z_{i}^{t} \in \{1,\dots, k\}$ denoting the label of the $i$-th vertex at level or time $t$. We also assume that, conditionally on the collection of latent groups $\{Z_{i}^{t}\}_{1 \leq t\leq T, 1\leq i \leq n}$, the graphs $\mathbf{A}^{1:T}$ are independent. For each fixed level or time $ t $, conditionally on the latent groups $ \{Z_{i}^{t}\}_{1 \leq i \leq n} $, the edges $ \{A_{ij}^{t}\}_{1 \leq i < j \leq n} $ are independent and each edge $ A_{ij}^{t} $ is a Bernoulli random variable with success probability $ P_{Z_{i}^{t}, Z_{j}^{t}}^t $, where $ (P_{ab}^t)_{1 \leq a,b \leq k} \in [0,1]^{k^{2}} $.
More precisely, we assume that 
\begin{equation*}A_{ij}^t \mid \{Z_i^t=a,Z_j^t=b\} \sim \mathcal{B}(P_{ab}^t),\end{equation*} 
with $P_{ab}^t=P_{ba}^t.$

To describe the complete likelihoods, we introduce the following counters for a latent configuration $\mathbf{z}_{n}$ and an observed graph $a_{n \times n}$.
\begin{equation*}n_{a}(\mathbf{z}_{n})= \sum_{i=1}^{n} \mathds{1}\{ z_{i}=a\},\end{equation*}
\begin{equation*}n_{ab}(\mathbf{z}_{n})= 
\begin{cases}
n_{a}(\mathbf{z}_{n}) n_{b}(\mathbf{z}_{n}), & 1\leq a<b \leq k, \\ 
\frac{1}{2} n_{a}(\mathbf{z}_{n}) (n_{a}(\mathbf{z}_{n})-1), & 1\leq a=b \leq k,
\end{cases}\end{equation*}
\begin{equation*}o_{a b}(\mathbf{z}_{n}, a_{n \times n})= 
\begin{cases}
\sum\limits_{1\leq i, j\leq n} \mathds{1}\left\{z_{i}=a, z_{j}=b\right\} a_{i j}, & a<b, \\ 
\sum\limits_{1\leq i < j\leq n}  \mathds{1}\left\{z_{i}=a, z_{j}=b\right\} a_{i j}, & a=b ,
\end{cases}\end{equation*}
and
\begin{equation*}c_{ab}(\znT)= \sum_{t=1}^{T-1}  \sum_{i=1}^{n} \mathds{1} \{z_i^{t}=a, z_i^{t+1}=b\}.\end{equation*}

In words, $n_{a}(\mathbf{z}_{n})$ (respectively $n_{ab}(\mathbf{z}_{n})$) is the number of nodes (resp. pairs of nodes) in cluster $a$ (resp. in clusters $(a,b)$), while $o_{a b}(\mathbf{z}_{n}, a_{n \times n})$ is the number of edges between clusters $a$ and $b$. These counters are defined for a single graph. Finally $c_{ab}(\znT)$ counts the number of transitions of nodes from clusters $a$ to cluster $b$ in the case of a sequence of graphs.

We now describe the specificities of each model.

\paragraph{Multi-layer stochastic block model}
We rely on the multi-layer stochastic block model (MLSBM) with $T$ levels or layers and $k$ communities as defined in \cite{def_multi_SBM}. We denote a $k$-multi-layer SBM a multi-layer stochastic block model with $k$ communities.

Here, the random variables $Z_i^t$ are independent of $t$, so each node $i$ is associated to a latent variable $Z_{i} \in \{1,\dots,k\}$ and is assigned to a class with probability $\pi=\left(\pi_{1},\dots,\pi_{k}\right)$. 
The distribution of the pair $\left(\mathbf{Z}_{n}, \AnT\right)$ is given by
\begin{align}\label{vraiss}
&\left( \mathbf{z}_{n},\anT\right) \mapsto  \mathbb{P}_{\pi, \Pvec}(\mathbf{z}_{n},\anT)\nonumber \\
&= \prod_{1 \leq a \leq k} \pi_{a}^{n_{a}}  \prod_{t=1}^{T} \prod_{1 \leq a \leq b \leq k} \left({P_{a b}^{t}} \right) ^ {o_{a b}^{t}}\left(1-{P_{a b}^{t}}\right)^{n_{a b}-o_{a b}^{t}},  
\end{align}
where $n_{ab}=n_{ab}(\mathbf{z}_{n})$ is the number of pairs of nodes in cluster $(a,b)$ across the layers and $o_{a b}^{t}= o_{a b}(\mathbf{z}_{n}, a_{n \times n}^{t})$ is the number of edges between clusters $a$ and $b$ at level $t$.
We denote by $\Theta^{k,T}_\ML$ the parametric space for a multi-layer model with $k$ communities and $T$ layers given by
\begin{align*}
\Theta^{k,T}_\ML=\{(\pi, \Pvec): & \pi \in(0,1)^{k}, \sum_{a=1}^{k} \pi_{a}=1, \forall t \in \{1,\dots, T\}, \\
& P^{t} \in[0,1]^{k \times k}, P^{t} \text { is symmetric}\} .
\end{align*}

\paragraph{Dynamic stochastic block model}
We rely on the dynamic stochastic block model (DynSBM) as defined in \cite{def_Dyn_SBM_MM}.

Here, the $\{Z_{i}^{1:T}\}_{1\leq i \leq n}$ are independent and identically distributed and each $Z_{i}^{1:T}=\left( Z_i^t\right)_{1 \leq t \leq T}$ is an irreducible, aperiodic and stationary Markov chain with transition probabilities
\begin{equation*}
\mathbb{P}(Z_{i}^{t+1} = b \mid Z_{i}^{t}=a) = \pi_{ab}, \quad \forall 1 \leq a,b\leq k,
\end{equation*}
where $\Pi=(\pi_{ab})_{1\leq a,b\leq k}$ is a stochastic matrix (i.e. with non-negative coefficients and with each row summing to one). We let $\alpha=(\alpha_{1},\dots, \alpha_{k})$ be the initial stationary distribution of the Markov chain. 
We then have
\begin{equation*}\mathbb{P}_{\Pi}(Z_{i}^{1:T})= \alpha_{Z_{i}^{1}} \prod_{t=1}^{T-1} \pi_{Z_{i}^{t}Z_{i}^{t+1}}, \quad \forall i\in [ 1,\dots, n ].\end{equation*}
Since, with discrete latent random variables, identifiability (i.e. it is possible to uniquely determine the model parameters based on the distribution of the observed data) 
can only be obtained up to label switching on the node groups for a permutation $\sigma$ which acts globally (meaning it is the same at each time point $t$), \cite{def_Dyn_SBM_MM} argue that it is necessary to add some constraints on the transition matrix $\Pi$ or on the parameter $\Pvec$. This allows to avoid identification problems caused by label switching at different times. So following these authors we choose to impose that for all $1 \leq a \leq k$ and $1 \leq t,t' \leq T$, we have
\begin{equation*}P_{aa}^t=P_{aa}^{t'}.\end{equation*} 
We denote by $\Theta^{k,T}_\dyn$ the parametric space for a dynamic SBM with $k$ communities, given by
\begin{align*}
&\Theta^{k,T}_\dyn=\{(\Pi, \Pvec): \text{ }\Pi \in[0,1]^{k \times k} \text{ is a stochastic matrix}, \\
& \qquad \forall t \in \{1,\dots, T\}, P^{t} \in[0,1]^{k \times k},  P^{t} \text { is symmetric,} \\
& \text{ and } \forall 1\leq a \leq k,\text{ and } \forall 1\leq t,t'\leq T, \text{ we have } P_{aa}^t=P_{aa}^{t'}\} .
\end{align*} 
Given the latent configuration at starting time $\mathbf{Z}^1_n$, the distribution of the pair $\left(\mathbf{Z}_{n}^{2:T}, \AnT\right)$ is given by
\begin{align}\label{vraiss_dyn}
&\left(\mathbf{z}_{n}^{2:T},\anT\right) \mapsto \mathbb{P}_{\Pi, \Pvec}((\mathbf{z}_{n}^{2:T},\anT)|\mathbf{z}^1_n) \nonumber \\
&=
   \prod_{1 \leq a, b \leq k} \pi_{ab}^{c_{ab}} \prod_{t=1}^{T} \prod_{1 \leq a \leq b \leq k} \left({P_{a b}^{t}} \right) ^ {o_{a b}^{t}}\left(1-{P_{a b}^{t}}\right)^{n_{a b}^t-o_{a b}^{t}}, 
\end{align}
where $n_{ab}^t =n_{ab}(\mathbf{z}_n^t)$, $o_{a b}^{t}= o_{a b}(\mathbf{z}_n^t, a_{n\times n}^t)$ are the number of nodes in cluster $a$ and the number of edges between clusters $a$ and $b$ at time $t$ respectively, and $c_{ab}=c_{ab}(\znT)$ is the number of node transitions from cluster $a$ to cluster $b$. Note that we consider a conditional distribution with respect to the state of the process at the initial time. This is standard when the latent variable is from a Markov chain (as for instance in Hidden Markov Models).

In this article we focus on both dense and sparse regimes. In the case of the dense regime the probability of having an edge is considered to be constant, which means that for all $1 \leq t \leq T$ the matrix $P^{0,t}$ does not depend on $n$. In the case of the sparse regime the probability of having an edge goes to zero with $n$ having order $\rho_n^t$ for the layer $t$, which means that for all $1 \leq t \leq T$ the matrix $P^{0,t}$ can be written as $\rho_n^t S^{0,t}$ with $S^{0,t}$ a matrix not depending on $n$ and with $n \rho_n^t \rightarrow \infty $ and $\rho_{n}^{t} \rightarrow 0$ for all $t \in \{1,\dots,T\}$. We denote $\mathbf{P}^0= \boldsymbol \rho_n \mathbf{S}^0 = \left( \rho_n^1 S^{0,1}, \dots ,  \rho_n^T S^{0,T}\right)$. 
In the multi-layer SBM, we allow a different sparsity for each layer (so $\rho_n^t$ indeed depends on $t$). However, in the context of dynamic SBMs, sparsity is not allowed to depend on time and we assume $\rho_n^t=\rho_n$ independent of time index $t$.

\section{The penalized Krichevsky–Trofimov estimator}\label{KT_est}

\subsection{Definition of the estimator}

We now define the penalized estimator of Krichevsky–Trofimov, which we use to estimate the order (i.e. the number of clusters) in the two models. First we define the order of a model.

\begin{definition} [Order of the model]
  The order of the multi-layer SBM (respectively dynamic SBM) is the smallest integer k for which the equality \eqref{vraiss} (resp. \eqref{vraiss_dyn}) holds for a pair of parameters $(\pi^{0} , \Pvec^{0} ) \in \Theta^{k,T}_\ML$ (resp. $(\Pi^{0} , \Pvec^{0} ) \in \Theta^{k,T}_\dyn$) and is denoted by $k_{0}$. 
 \end{definition}

\begin{ass}\label{ass1}
	If a multi-layer or dynamic SBM has order $k_{0}$ then it cannot be reduced to a model with less communities than $k_{0}$. We therefore assume that there exists $1 \leq t \leq T$ such that the matrix $P^{0,t}$ does not have two identical columns (or rows).  
\end{ass}

\begin{rque}
Assumption \ref{ass1} excludes the possibility that the communities cannot be distinguished in any layer. 
It is true that one can construct artificial examples where, for every $t$, two rows or columns of $P^{0,t}$ coincide, but with the pair of indistinguishable communities depending on $t$. 
Such situations require very specific parameter configurations and are therefore non-generic: if one selects the parameters randomly in the admissible space, the probability of encountering this structure is zero. 
Hence, Assumption \ref{ass1} is not restrictive in practice, as it only rules out these pathological cases, which are unlikely to occur in real applications.
\end{rque}

Before defining the penalized Krichevsky–Trofimov estimator, we define a prior distribution on the parameters and a penalty for each model.

For the multi-layer SBM, we choose as a prior distribution the product of a $\operatorname{Dirichlet}(1 / 2, \dots, 1 / 2)$ for $\pi$, and a product of $Tk(k+1) / 2$ independent $ \operatorname{Beta}(1 / 2,1 / 2)$ distributions for $\Pvec$. Formally, we thus define the distribution $\nu_{k}(\pi, \Pvec)$ on $\Theta^{k,T}_\ML$ as
\begin{align}\label{prior_ml}
\nu_{k}(\pi, \Pvec)= & \frac{\Gamma(\frac{k}{2})}{\Gamma(\frac{1}{2})^{k}} \prod_{1 \leq a \leq k} \pi_{a}^{-\frac{1}{2}} \nonumber \\
&\times \prod_{t=1}^{T} \prod_{1 \leq a \leq b \leq k} \frac{1}{\Gamma(\frac{1}{2})^{2}} {(P_{a b}^{t}})^{-\frac{1}{2}}(1-P_{a b} ^ {t})^{-\frac{1}{2}}, 
\end{align}
and for the penalty, we let
\begin{align}
&\pen_\ML(k, n,T) \nonumber \\
&=  \sum_{i=1}^{k-1} \left( \frac{1}{2} \frac{Ti(i+1)}{2}  \log (n^2) + \frac{1}{2} (i-1) \log n + (1+\epsilon) \log n \right) \nonumber \\
&=  \sum_{i=1}^{k-1} \left[\frac{Ti(i+1)+i-1}{2} +1+\epsilon \right] \log n ,
\label{pen}
\end{align}
for some $\epsilon >0$. The first term $1/2 \times Ti(i+1)/2$ accounts for half the number of connectivity parameters $P_{ab}$ for $a \leq b$ and factors $\log(n^2)$ where $n^2$ is the order of such possible interactions. The second term $(i-1)/2$ is half the dimension of the class proportions $\pi$ and appears times $\log n$. A constant larger than 1 (in practice $1 + \epsilon$) times $\log n$ is added to ensures that the estimator does not overestimate the right number of clusters. This is the classic form of penalties for the KT estimator. Its structure directly follows from the proof and that penalty is certainly too large. However, we note that for $T=1$, we obtain exactly the penalty given in \cite{CL20}. While this penalty might not be minimal, we stress that our result is the first consistency result for estimating the order of a collection of SBMs.

For the dynamic SBM, we choose as a prior distribution a product of $k$ independent Dirichlet $(1/2, \dots, 1/2)$, on each row of $\Pi$, and a product of $k+Tk(k-1) / 2$ independent $\operatorname{Beta}(1 / 2,1 / 2)$ distributions for $\Pvec$. Formally, we thus define the distribution $\nu_{k}(\Pi, \Pvec)$ on $\Theta^{k,T}_\dyn$ as
\begin{align} \label{prior_dyn}
\nu_{k}(\Pi, \Pvec)= &\prod_{1\leq b\leq k} \left[ \frac{\Gamma(\frac{k}{2})}{\Gamma(\frac{1}{2})^{k}} \prod_{1 \leq a \leq k} \pi_{ab}^{-\frac{1}{2}} \right] \nonumber \\
& \times \prod_{1\leq a \leq k}\frac{1}{\Gamma(\frac{1}{2})^{2}} (P_{a a})^{-\frac{1}{2}}(1-P_{a a} )^{-\frac{1}{2}} \nonumber\\
& \times \prod_{t=1}^{T} \prod_{1 \leq a < b \leq k} \frac{1}{\Gamma(\frac{1}{2})^{2}} (P_{a b}^{t})^{-\frac{1}{2}}(1-P_{a b} ^ {t})^{-\frac{1}{2}},
\end{align}
and for the penalty, we let
\begin{align}
& \pen_\dyn(k, n,T) \nonumber \\
&= \sum_{i=1}^{k-1} \biggl( \frac{1}{2} i \log(n^2 T) + \frac{1}{2}i(i-1) \log(nT) \nonumber \\
&\hspace{3cm}+ \frac{1}{2} \frac{Ti(i-1)}{2}\log(n^2) + (1+\epsilon) \log n  \biggr) \nonumber \\
& =\sum_{i=1}^{k-1} \frac{i}{2} \log (n^2T) +  \frac{i(i-1)}{2} \log (nT) \nonumber \\
& \hspace{3cm}+\left[\frac{Ti(i-1) }{2}+1+\epsilon \right] \log n,
\label{pen_dyn}
\end{align}
for some $\epsilon>0$. The penalty is also in the classic form of the penalties for the KT estimator. Indeed, the first term $i/2$ accounts for half the number of connectivity parameters $P_{aa}$ and factors $\log(n^2T)$ where $n^2T$ is the order of such possible interactions. The second term $i(i-1)/2$ is half the dimension of the stochastic matrix $\Pi$ and appears times $\log(nT)$ where $nT$ is the total number of possible transitions. The third term $Ti(i-1)/2$ stands for the number of connectivity parameters $P_{ab}^t$ for $a<b$ and appears times $\log(n^2)/2$ where $n^2$ is the order of possible interactions per layers. Finally the last term, a constant larger than 1 (in practice $1 + \epsilon$) times $\log n$, is the penalty term that ensures the estimator does not overestimate the right number of clusters. The structure of the penalty also results from the proof, and it is certainly too large\blue{---for instance, with respect to a BIC-like penalty.} 

We can now define the estimator: first in a multi-layer SBM, then in a dynamic SBM. The only difference between the two definitions of the estimators lies in the conditioning added in the integrated likelihood for the dynamic SBM model, as well as in the penalization.

\begin{definition} [Penalized KT estimator for MLSBM] 
The penalized Krichevsky–Trofimov estimator of the number of communities in a multi-layer SBM is defined by
\begin{equation}
\hat{k}_{\text{KT}}(\AnT)=\underset{1 \leq k \leq n}{\arg \max } \{\log \mathbf{KT}_{k}^{T}(\AnT)-\pen_\ML(k, n,T)\}, 
\label{KTdef}
\end{equation}
with the penalty defined in \eqref{pen} and where $\mathbf{KT}_{k}^{T}(\AnT)$ is the integrated likelihood with respect to the prior distribution defined in \eqref{prior_ml}, given by
\begin{align*}
\mathbf{KT}_{k} ^ {T}(\AnT) &= \mathbb{E}_{\nu_{k}}\left[\mathbb{P}_{\pi, \Pvec}(\AnT)\right]  \\
&= \int_{\Theta^{k,T}_\ML} \mathbb{P}_{\pi,\Pvec}(\AnT) \nu_{k}(\pi, \Pvec) d \pi d \Pvec.
\end{align*}
\end{definition}

\begin{definition} [Penalized KT estimator for DynSBM]
The penalized Krichevsky–Trofimov estimator of the number of communities in a dynamic SBM is defined by
\begin{equation}
\hat{k}_{\text{KT}}(\AnT )=\underset{1 \leq k \leq n}{\arg \max } \{\log \mathbf{KT}_{k}^{T}(\AnT |\mathbf{Z}_ n^1)-\pen_\dyn(k, n,T)\}, 
\label{KTdef_dyn}
\end{equation}
with the penalty defined in \eqref{pen_dyn} and where $\mathbf{KT}_{k}^{T}(\AnT|\mathbf{Z}_ n^1)$ is the integrated likelihood with respect to the prior distribution defined in \eqref{prior_dyn}, given by
\begin{align*}
 \mathbf{KT}_{k} ^ {T}(\AnT|\mathbf{Z}_n^1) &= \mathbb{E}_{\nu_{k}}\left[\mathbb{P}_{\Pi, \Pvec}(\AnT |\mathbf{Z}_n^1)\right]  \\
& = \int_{\Theta^{k,T}_\dyn} \mathbb{P}_{\Pi,\Pvec}(\AnT |\mathbf{Z}_n^1) \nu_{k}(\Pi, \Pvec) d \Pi d \Pvec.
\end{align*}
\end{definition}

Note that there is no a priori known upper bound on the order of the model and the $\arg \max$ appearing in \eqref{KTdef} and \eqref{KTdef_dyn} is taken for $k \in \{1,\dots,n\}$.

\begin{rque}
A natural alternative strategy would be to apply a consistent single-layer estimator (such as KT of \citep{CL20}) separately on each layer and then retain the maximum number of communities across layers. However, this approach only benefits from the sample size of a single layer, while our estimator pools information across all layers. As shown in Lemma \ref{lemme_exp}, this results in a faster non-overestimation rate depending on $nT$ rather than $n$ alone. In addition, numerical experiments confirm that the proposed multi-layer estimator outperforms this layer-wise baseline.
\end{rque}

\begin{rque}
\blue{The implementation of the estimator in the DynSBM framework is still an open problem; see the simulation section for further details and discussion.}
\end{rque}

\subsection{The consistency theorem}

We now introduce the consistency theorem for the penalized KT estimator.

\begin{theorem} \label{thm}
Consider the multi-layer SBM with $T$ layers and $k_{0}$ communities (resp. the dynamic SBM with $T$ times point and $k_{0}$ communities). Let $\hat k$ be defined as in \eqref{KTdef} (resp. \eqref{KTdef_dyn}). Then, under both sparse (with $n\rho_n^t = \Omega(\log n)$ for all $1 \leq t \leq T$) and dense regimes, we have:
\begin{equation*}
\hat{k}_{K T}(\AnT)=k_{0},
\end{equation*}
eventually almost surely as $n \rightarrow \infty$, with $T$ remaining fixed.
\end{theorem}

The following remark clarifies the distinction between weak recovery and exact recovery, 
and explains why the stronger regime is required.

\begin{rque}
When $n\rho_n^t \to \infty$, one can achieve weak recovery, in the sense that 
the estimated partition is positively correlated with the true partition and 
an increasing fraction of nodes are correctly classified with high probability. 
\blue{Although the analysis could have been focused only on the giant component for the $n\rho_n \rightarrow \infty$ case below $\log(n)$ regime, it is imperative to tackle the problem first for the fully connected networks.}
\end{rque}

\blue{Since the true data-generating mechanism of a multilayer network is generally unknown, it is natural to question the robustness of the proposed penalizations beyond their respective modeling frameworks, a point we briefly address below.}

\begin{rque}
\blue{The penalization introduced in Equation~\eqref{pen} is specifically derived for the MLSBM framework, while the penalization in Equation~\eqref{pen_dyn} is tailored to the DynSBM model. These penalties exploit the particular dependence structures of each model and are essential to obtain the corresponding consistency results.
In practice, however, the true data-generating process underlying a multilayer network is typically unknown. This naturally raises the question of robustness to model misspecification, for instance when data generated from a DynSBM are estimated using the MLSBM criterion (or conversely).
From a theoretical standpoint, crossing the penalizations would prevent the current proofs from carrying over directly, since they rely on the precise form of each model. Nevertheless, given the close relationship between the MLSBM and DynSBM frameworks, some degree of robustness may reasonably be expected in practice. An empirical illustration of this behavior under model misspecification is provided in Appendix~\ref{app:misspecification}.}
\end{rque}

\section{Key results and global structure of the proof of the consistency theorem} \label{proof_thm}

We now give the general structure of the proof of Theorem \ref{thm} stating all the key results that lead to the consistency theorem. The proofs of these intermediate results are postponed to the Appendix. The proof is divided into two parts. The first one, presented in Section \ref{non_over}, proves that the estimator does not overestimate the true order $k_0$. The second part, presented in Section \ref{non_under}, shows that the estimator does not underestimate the true order. We start by stating a proposition that is fundamental into both parts.

\subsection{Uniform bound for the likelihood function in terms of the KT distribution} 

\begin{prop} \label{propfond}
   For all $T$, all $k$, all $n \geq k$ and all $\AnT$ we have that for a multi-layer SBM
\begin{align}
&\log \mathbf{KT}_{k}^{T}(\mathbf{A}_{n \times n} ^ {1:T})  \leq \log \sup _{(\pi,  \Pvec) \in \Theta^{k,T}_\ML} \mathbb{P}_{\pi, \Pvec}(\AnT) \nonumber \\
& \leq \log \mathbf{KT}_{k}^{T}(\mathbf{A}_{n \times n} ^ {1:T})+\frac{Tk(k+1)+k-1}{2} \log n +c_{k,T}, 
 \label{eqprop1}
\end{align}
where $ c_{k,T}=Tk(k+1)+1,$ and for a dynamic SBM, for all initial configuration $\mathbf{z}_n^1 \in [k]^n$ we have that
\begin{align}
&\log \mathbf{KT}_{k}^{T}(\AnT |\mathbf{z}_n^1)   \leq \log \sup _{(\Pi,  \Pvec) \in \Theta^{k,T}_\dyn} \mathbb{P}_{\Pi, \Pvec}(\AnT |\mathbf{z}_n^1)
\nonumber \\
& \leq \log \mathbf{KT}_{k}^{T}\left(\AnT |\mathbf{z}_n^1\right)+ \frac{k}{2} \log(n^2T)+\frac{k(k-1)}{2}  \log (nT) \nonumber \\
&\hspace{4cm}+\frac{Tk(k-1)}{2} \log n+c_{k,T}, 
\label{eqprop1_dyn}
\end{align}
where
$
c_{k,T}=\frac{k}{3T}[k(k-1)+2]+Tk(k-1)+2k
$.
\end{prop}

The proof can be found in Appendix \ref{preuve_propfond}.

\subsection{Non-overestimation in the consistency proof} \label{non_over}

We now establish that the estimator $\hat{k}_{KT}$ does not overestimate the true number of communities $k_0$. We start by proving a lemma that is useful to bound the probability of overestimation.

\begin{lemme} \label{lemme_exp}
For $k>k_{0}$ we have for a multi-layer SBM with order $k_0$ and parameter ($\pi^{0}, \Pvec^{0}$)
\begin{align*}
   & \mathbb{P}_{\pi^{0}, \Pvec^{0}}(\hat{k}_{K T}\left(\AnT\right)=k) \\
&\leq \exp \left\{\frac{Tk_{0}(k_{0}+1)+k_0-1}{2} \log n+c_{k_{0},T}+d_{k_{0}, k, n,T}^\ML\right\},
\end{align*}
and for a dynamic SBM with order $k_0$ and parameter ($\Pi^{0}, \Pvec^{0}$)
\begin{align*}
   & \mathbb{P}_{\Pi^{0}, \Pvec^{0}}(\hat{k}_{K T}\left(\AnT \right) =k)\\
&\leq \exp \biggl\{ \frac{k_0}{2}\log(n^2T)+\frac{k_0(k_0-1)}{2}  \log (nT) \\
&\hspace{2cm}+\frac{Tk_0(k_0-1)}{2} \log n+c_{k_0,T}  +d_{k_{0}, k, n,T}^\dyn \biggr\},
\end{align*}
where $d_{k_{0}, k, n,T}^\text{u}=\pen_\text{u}\left(k_{0}, n,T\right)-\pen_\text{u}(k, n,T)$ and $\text{u} \in \{\ML,\dyn\}$.
\end{lemme}

The proof can be found in Appendix \ref{preuve_lemme_exp}.
We now demonstrate that $\hat{k}_{KT}$ does not overestimate $k_0$, which means that we want to establish that $\hat{k}_{KT} \leq k_0$ eventually almost surely. To this aim we start by establishing that, eventually almost surely when $n \rightarrow \infty$ the estimator cannot take values in the interval ($k_0$, $\log n$] and then we prove that it is not greater than $\log n $ either.

\begin{lemme} \label{lemme_pas_k,log}
Let $\AnT$ be a sample of size $n$ and order $k_{0}$ from either a multi-layer SBM or a dynamic SBM with $T$ graphs. We have that
\begin{equation*}
\hat{k}_{K T}\left(\AnT\right) \notin\left(k_{0}, \log n\right],
\end{equation*}
eventually almost surely when $n \rightarrow \infty$.
\end{lemme}

The proof can be found in Appendix \ref{preuve_lemme_pas_k,log}.

\begin{lemme} \label{lemma_pas_log,n}
Let $\AnT$ be a sample of size $n$ and order $k_{0}$ from either a multi-layer SBM or a dynamic SBM with $T$ graphs. We have that
\begin{equation*}
\hat{k}_{K T}\left(\AnT\right) \notin(\log n, n],
\end{equation*}
eventually almost surely when $n \rightarrow \infty$.
\end{lemme}

The proof can be found in Appendix \ref{preuve_lemma_pas_log,n}.
By combining  Lemmas \ref{lemme_pas_k,log} and \ref{lemma_pas_log,n} we immediately obtain that $\hat{k}_{KT}$ does not overestimate $k_0$, wich is formally stated in the next result.

\begin{prop} [Non-overestimation]
Let $\AnT$ be a sample of size $n$ from a multi-layer SBM of order $k_{0}$ with $T$ layers (resp. a dynamic SBM with $k_0$ communities and $T$ time points). Then, the $\hat{k}_{k T}\left(\AnT\right)$ order estimator defined in \eqref{KTdef} (resp. \eqref{KTdef_dyn}) does not overestimate $k_{0}$, eventually almost surely when $n \rightarrow \infty$.
\end{prop}

\subsection{Non-underestimation in the consistency proof} \label{non_under}

We now prove that the estimator $\hat{k}_{KT}$ does not underestimate the true number of communities $k_0$. We start by defining the profile likelihood estimator of the label assignment as well as the confusion matrix $Q_n$.

\begin{definition} 
 The profile likelihood estimator of the label assignment under the $k$-multi-layer stochastic block model, with $\AnT$ a sample from a $k_{0}$-multi-layer SBM, is defined as 
 \begin{equation} \label{def_z_etoile}
 \mathbf{z}_{n,k}^{\star}=\underset{\mathbf{z}_{n} \in\left[k\right]^{n}}{\arg \max } \sup _{(\pi, \Pvec) \in \Theta^{k,T}_\ML} \mathbb{P}_{\pi, \Pvec}\left(\mathbf{z}_{n}, \AnT\right).
 \end{equation}
 In the case of a dynamic block model with $k$ communities, with $\AnT$ a sample from a $k_{0}$-dynamic SBM, the profile likelihood estimator is defined as 
 \begin{equation*}  
 \mathbf{z}_{n,k}^{\star}=\underset{\mathbf{z}_n^{2:T} \in\left[k\right]^{nT}}{\arg \max } \sup _{(\pi, \Pvec) \in \Theta^{k,T}_\dyn} \mathbb{P}_{\pi, \Pvec}\left(\mathbf{z}_n^{2:T} , \AnT | \mathbf{z}_{n}^1\right).
 \end{equation*}
\end{definition}

\begin{definition}[Confusion matrix]
    Define for all $\mathbf{\bar z} \in\{1,\dots, k\}^n$  and  $\mathbf{z}\in \{1,\dots, k_0\}^n$ the  $k\times k_0$ matrix $Q_n(\mathbf{\bar z},\mathbf{z})$ given 
by 
\begin{equation*}
[Q_n(\mathbf{\bar z},\mathbf{z})]_{aa'} \;= \frac{1}{n} \sum_{i=1}^n \mathds{1}\{\bar z_i=a, z_i=a'\}.
\end{equation*}
Observe that $n_{a'}(\mathbf{z}) =  n[Q_n^{\intercal}(\mathbf{\bar z},\mathbf{z}){\bf{1}}_{k}]_{a'}$, with ${\bf1}_{k}$  a column vector of dimension $k$
with all entries equal to 1.
Moreover, the matrix $Q_n(\mathbf{\bar z},\mathbf{z})$ satisfies 
\begin{equation*}
\| Q_n(\mathbf{\bar z},\mathbf{z})\|_1 = \sum_{a=1}^k\sum_{a'=1}^{k_0}  [Q_n(\mathbf{\bar z},\mathbf{z})]_{aa'} = 1,
\end{equation*}
for all $(\mathbf{\bar z},\mathbf{z})$ and
\begin{equation*}
\mathbb{E}_{\theta^0_{\text{mod}}}[\tilde o_{ab}( \mathbf{\bar z},A_{n \times n}^{t}) \mid \mathbf{Z}=\mathbf{z}]= n^2[Q_n(\mathbf {\bar z},\mathbf{z}) P^{0,t} Q_n(\mathbf{\bar z},\mathbf{z})^{\intercal}]_{ab},
\end{equation*}
where $\theta^0_{\text{mod}} =( \pi^0, \Pvec^{0}) $ if $\text{mod} =\ML$ and $\theta^0_{\text{mod}} =( \Pi^0, \Pvec^{0}) $ if $\text{mod} = \dyn$, and where
\begin{equation} \label{o_tilde}
\tilde{o}_{ab}( \mathbf{\bar z},A_{n \times n}^{t})=\sum_{1 \leq i, j \leq n} \mathds{1}\left\{\bar z_{i}=a, \bar z_{j}=b\right\} a_{i j}^{t},
\end{equation}
for all pairs $a,b$ and notice that $\tilde{o}_{ab}=o_{ab}$ for all $a \neq b$, and $\tilde{o}_{aa}=2o_{aa}$ for all $a$.
\end{definition}

We present the results useful for proving the non-underestimation in the case of a sparse regime (where for all $1 \leq t \leq T$ the matrix $P^{0,t}$ can be written as $\rho_n^t S^{0,t}$ with $S^{0,t}$ a matrix not depending on $n$), but we draw the reader's attention to the fact that all the following lemmas and propositions can be transposed in the dense case by taking  $\rho_n^t =1$ and $S^{0,t}= P^{0,t}$ for all $1 \leq t \leq T$, and by replacing the function $\tau(u)=u \log u -u$ by the function $\gamma(u)=u \log u + (1-u)\log (1-u)$ (the density of the Bernoulli distribution).

First we establish a concentration lemma for the observed number of edges between any two clusters in the complete model (the model where we know all the information, which means the groups and the connection probabilities), uniformly over the fixed label assignment. This lemma is used to establish the following four lemmas.

\begin{lemme} [Concentration of the observed counts]  \label{lemme_prel_1}
 Consider a $k_{0}$-multi-layer SBM with $T$ layers and parameters $\left(\pi^{0}, \Pvec^{0}\right)$, with $ \Pvec^{0} = \boldsymbol \rho_{n} \Svec^{0}$ with $\Svec^{0}$ not depending on $n$, and $ \boldsymbol \rho_{n} =(\rho_{n}^{1}, \dots, \rho_{n}^{T})$ with $\rho_n^t \geq C\log n/n$ where $C$ is a
sufficiently large constant not depending on n, and $\rho_{n}^{t} \rightarrow 0$ for all $t \in \{1,\dots,T\}$. For any $\xi>0$  and $a,b\in [k]$, and for any $1 \leq t \leq T$ we have that, with $S^{0,t}_{\max}= \max_{a,b}S^{0,t}_{ab}$,
\begin{align*}
&\mathbb P_{\pi^0, \mathbf{P}^0}\Bigl(\sup_{\mathbf{\bar z}\in [k]^n}\; \Bigl|\frac{\tilde o_{ab}(\mathbf{\bar z},A_{n\times n}^{t})}{ \rho_n^t n^2} \\
&\hspace{3cm}- [Q_n(\mathbf{\bar z},\mathbf Z_n) S^{0,t} Q_n(\mathbf{\bar z},\mathbf Z_n)^{\intercal}]_{ab}\Bigr|>\xi\Bigr) \\
&\leq 2 \exp\left(- \frac{\rho_n^t n^{2} \xi ^2}{4 S^{0,t}_{\max}  + \frac{4}{3}\xi} + n \log k \right).
\end{align*}
 Consider a dynamic SBM with $k_{0}$ communities and parameters  $\left(\Pi^{0}, \Pvec^{0}\right)$, with $\Pvec^{0}=\rho_n \Svec^0$ with $\Svec^0$ not depending on $n$, with $\rho_n \geq C \log n/n$ and $\rho_n \rightarrow 0$. For any $\xi>0$  and $a \neq b\in [k]$, and for any $1 \leq t \leq T$ we have that, with $S^{0,t}_{\max}= \max_{a,b}S^{0,t}_{ab}$,
\begin{align*}
&\mathbb P_{\Pi^0, \mathbf{P}^0}\Bigl(\sup_{\mathbf{\bar z}_n^{t}\in [k]^{n}} \Bigl|\frac{ o_{ab}(\mathbf{\bar z}_n^{t},A_{n\times n}^t)}{ \rho_n n^2} \\
&\hspace{3cm}- [Q_n(\mathbf{\bar z}_n^{t},\mathbf{Z}_n^t) S^{0,t} Q_n(\mathbf{\bar z}_n^{t},\mathbf{Z}_n^t)^{\intercal}]_{ab}\Bigr|>\xi \Bigr)\\
& \leq 2 \exp \Big[- \frac{\rho_n n^{2} \xi ^2}{4 S^{0,t}_{\max}  + \frac{4}{3}\xi} + nT \log k\Big],
\end{align*}
and 
\begin{align*} 
&\mathbb P_{\Pi^0, \mathbf{P}^0}\Bigl(\sup_{\mathbf{\bar z}_n^{1:T}\in [k]^{nT}} \Bigl|\frac{ \sum_{t=1}^T o_{aa}(\mathbf{\bar z}_n^{t},A_{n\times n}^t)}{ \rho_n n^2}\\
& \hspace{2cm}- \frac{1}{2}\sum_{t=1}^T[Q_n(\mathbf{\bar z}_n^{t},\mathbf{Z}_n^t)  S^{0,t} Q_n(\mathbf{\bar z}_n^{t},\mathbf{Z}_n^t)^{\intercal}]_{aa}\Bigr|>\xi \Bigr) \\
&\leq 2 \exp \Big[\frac{- \rho_n n^2 \xi ^2 }{4TS^{0}_{\max}  + \frac{4}{3}\xi}+ nT \log k\Big],
\end{align*}
 with $S^{0}_{\max}= \max_{t}S^{0,t}_{\max}$.
\end{lemme}

 The proof can be found in Appendix \ref{preuve_lemme_prel_1}.
We now introduce some lemmas that are useful in the proof of the non-underestimation of $k_0$ by $\hat{k}_{KT}$. We start with two lemmas in the case of the multi-layer SBM. Lemma \ref{lemme_prel_2} replaces a law of large number for the observed edge counts between any pairs of clusters under a configuration $\mathbf{z}_{n,k}^{\star}$ obtained in a smaller model $(k<k_0)$.

\begin{lemme} \label{lemme_prel_2}
 Consider a $k_{0}$-multi-layer SBM with $T$ layers and parameters $\left(\pi^{0}, \Pvec^{0}\right) \in \Theta_\ML^{k_0,T}$, with $ \Pvec^{0} =\boldsymbol \rho_{n} \Svec^{0}$ with $\Svec^{0}$ not depending on $n$, and $\boldsymbol \rho_{n} =(\rho_{n}^{1}, \dots, \rho_{n}^{T})$ with $\rho_n^t \geq C\log n/n$ and $\rho_{n}^{t} \rightarrow 0$ for all $t \in \{1,\dots,T\}$. For $1 \leq t \leq T$, and $k<k_0$  there exists a $k\times k$ positive matrix $S^{\star,t}$ and a $k$-dimensional vector $\pi^\star$ such that 
\begin{align*}
    \limsup_{n\rightarrow \infty} \frac{1}{2}  \sum_{1 \leq a,b \leq k} &\frac{n_{a}( \mathbf{z}_{n,k}^{\star}) n_{b}( \mathbf{z}_{n,k}^{\star})}{ n^{2}} \tau \left(\frac{\tilde{o}_{a b}( \mathbf{z}_{n,k}^{\star}, A_{n \times n}^{t})}{\rho_n^t n_{a}( \mathbf{z}_{n,k}^{\star}) n_{b}( \mathbf{z}_{n,k}^{\star})}\right) \\
    &\leq  
    \frac{1}{2} \sum_{1 \leq a,b \leq k} \pi_{a}^{\star} \pi_{b}^{\star} \tau \left(S^{\star,t}_{ab}\right),
\end{align*}
almost surely. Moreover,  for all $1 \leq t \leq T$, the pair $(\pi^\star,S^{\star,t})$ is given by 
\begin{align} \label{def_S_star_pi_star} 
    \pi^\star_a &=[R^{\star}{\bf{1}}_{k_0}]_a, \quad   a\in \{1,\dots,k\} \nonumber\\
    S^{\star,t}_{ab} &= \frac{[R^{\star}S^{0,t}(R^{\star})^{\intercal}]_{ab}}{[R^{\star}{\bf{1}}_{k_0}{\bf{1}}_{k_0}^{\intercal} (R^{\star})^{\intercal}]_{ab}},\quad a,b\in \{1,\dots,k\},
\end{align}
for  a matrix  $R^{\star} \in [0,1]^{k \times k_{0}}$  satisfying  $\|R^{\star}\|_1=1$ and having one and only one non-zero entry on each column. 
\end{lemme}

The proof can be found in Appendix \ref{preuve_lemme_prel_2}.

\begin{lemme} 
[Lower bound on the log-maximum-likelihood ratio - MLSBM] \label{lemme_prel_3}
Consider a $k_{0}$-multi-layer SBM with $T$ layers and parameters $\left(\pi^{0}, \Pvec^{0}\right)$, with $ \Pvec^{0} = \boldsymbol \rho_{n} \Svec^{0}$ with $\Svec^{0}$ not depending on $n$, and $\boldsymbol  \rho_{n} =(\rho_{n}^{1}, \dots, \rho_{n}^{T})$ with $\rho_n^t \geq C\log n/n$ and $\rho_{n}^{t} \rightarrow 0$ for all $t \in \{1,\dots,T\}$. Then for all $k = k_{0}-1$ and $(\pi^{\star}, S^{\star,t})$ as in Lemma \ref{lemme_prel_2} we have that 
\begin{equation*}
    \sum_{1 \leq a,b \leq k_{0}} \pi_{a} \pi_{b} \tau \left(S^{0,t}_{ab}\right) - \sum_{1 \leq a,b \leq k} \pi_{a}^{\star} \pi_{b}^{\star} \tau \left(S^{\star,t}_{ab}\right) \geq 0,
\end{equation*}
with a strict inequality if $S^{0,t}$ has no two identical columns.
\end{lemme}

The proof can be found in Appendix \ref{preuve_lemme_prel_3}.
Then we introduce similar Lemmas for the case of the dynamic SBM.

 \begin{lemme}\label{lemme_prel_2_dyn}
Consider a dynamic SBM with $k_{0}$ communities and parameters  $\left(\Pi^{0}, \Pvec^{0}\right)$, with $\Pvec^{0}=\rho_n \Svec^0$ with $\Svec^0$ not depending on $n$, with $\rho_n \geq C \log n/n$ and $\rho_n \rightarrow 0$. For $1 \leq t \leq T$, and $k<k_0$  there exists a $k\times k$ positive matrix $S^{\star,t}$ and a $k$-dimensional vector $\alpha^\star$ such that 
\begin{align*}
    \limsup_{n\rightarrow \infty} \frac{1}{2} \sum_{t=1}^T  \sum_{1 \leq a \neq b \leq k} & \frac{n_{ab}(\mathbf{z}_{n,k}^{\star,t}) }{ n^{2}} \tau \left( \frac{{o}_{a b}(\mathbf{z}_{n,k}^{\star,t}, A_{n\times n}^t)}{\rho_n n_{ab}(\mathbf{z}_{n,k}^{\star,t})}\right) \\& \leq  
    \frac{1}{2}\sum_{t=1}^T \sum_{1 \leq a \neq b \leq k} \alpha_{a}^{\star} \alpha_{b}^{\star} \tau \left(S^{\star,t}_{ab}\right),
\end{align*}
and 
\begin{align*}
    \limsup_{n\rightarrow \infty} \sum_{1 \leq a \leq k} &\frac{\sum_{t=1}^T n_{aa}(\mathbf{z}_{n,k}^{\star,t})}{n^2} \tau \left( \frac{\sum_{t=1}^To_{aa}(\mathbf{z}_{n,k}^{\star,t},A_{n\times n}^t)}{\sum_{t=1}^T\rho_n n_{aa}(\mathbf{z}_{n,k}^{\star,t})}\right)\\& \leq \frac{T}{2}\sum_{1 \leq a \le k} (\alpha_{a}^{\star})^2  \tau \left(S^{\star}_{aa}\right).
\end{align*}
Moreover,  $(\alpha^\star,S^{\star,t})$ is given by 
\begin{align} \label{def_S_star_alpha_star}
    \alpha^\star_a &=[R^{\star}{\bf{1}}_{k_0}]_a, \quad   a\in \{1,\dots,k\} \nonumber\\
    S^{\star,t}_{ab} &= \frac{[R^{\star}S^{0,t}(R^{\star})^{\intercal}]_{ab}}{[R^{\star}{\bf{1}}_{k_0}{\bf{1}}_{k_0}^{\intercal} (R^{\star})^{\intercal}]_{ab}},\quad a \neq b\in \{1,\dots,k\}, \nonumber \\
    S^{\star}_{aa} &= \frac{\sum_{t=1}^T[R^{\star}S^{0,t}(R^{\star})^{\intercal}]_{aa}}{T[R^{\star}{\bf{1}}_{k_0}]_{a}^2},\quad a \in \{1,\dots,k\}.
\end{align}
for a matrix  $R^{\star} \in [0,1]^{k \times k_{0}}$  satisfying  $\|R^{\star}\|_1=1$ and having one and only one non-zero entry on each column. 
 \end{lemme}

 The proof is similar to the proof of Lemma \ref{lemme_prel_2} and it is given in Appendix \ref{preuve_lemme_prel_2_dyn}.

 \begin{lemme} [Lower bound on the log-maximum-likelihood ratio - dynSBM]  \label{lemme_prel_3_dyn}
Consider a dynamic SBM with $k_{0}$ communities and parameters  $\left(\Pi^{0}, \Pvec^{0}\right)$, with $\Pvec^{0}=\rho_n \Svec^0$ with $\Svec^0$ not depending on $n$, with $\rho_n \geq C \log n/n$ and $\rho_n \rightarrow 0$, and with $\alpha$ the initial stationary distribution of the Markov chain. Assume that there exist a $t$ such that $S^{0,t}$ has no two identical columns. Then for $k = k_{0}-1$ and $(\alpha^{\star}, S^{\star,t})$ as in Lemma \ref{lemme_prel_2_dyn} we have that 
\begin{align*}
    &\frac{1}{2} \sum_{t=1}^T   \sum_{1 \leq a\neq b \leq k_{0}} \alpha_{a} \alpha_{b} \tau \left(S^{0,t}_{ab}\right)  + \frac{T}{2}  \sum_{1 \leq a \leq k_{0}}  \alpha_{a}^2 \tau \left(S^{0}_{aa}\right) \\
    &- \frac{1}{2} \sum_{t=1}^T\sum_{1 \leq a \neq b \leq k} \alpha_{a}^{\star} \alpha_{b}^{\star} \tau \left(S^{\star,t}_{ab}\right) - \frac{T}{2}\sum_{1 \leq a \le k}  (\alpha_{a}^{\star})^2  \tau \left(S^{\star}_{aa}\right)>0.
\end{align*}
 \end{lemme}

The proof is similar to the proof of Lemma \ref{lemme_prel_3} and it is given in Appendix \ref{preuve_lemme_prel_3_dyn}.
We are now able to prove that $\hat{k}_{KT}$ does not underestimate $k_0$.

\begin{prop} [Non-underestimation] 
Assume that $(\mathbf{Z}_{n},\AnT)$ is a sample of size $n$ from a multi-layer SBM of order $k_{0}$ with $T$ layers  with parameters $\left(\pi^{0}, \Pvec^{0}\right)$, (resp. $(\ZnT,\AnT)$ a sample of size $n$ from a dynamic SBM of order $k_{0}$ with parameters  $\left(\Pi^{0}, \Pvec^{0}\right)$, and with $\alpha$ the initial stationary distribution of the Markov chain), with $ \Pvec^{0} = \boldsymbol \rho_{n} \Svec^{0}$ with $\Svec^{0}$ not depending on $n$, and $\boldsymbol \rho_{n} =(\rho_{n}^{1}, \dots, \rho_{n}^{T})$ (resp. $\rho_n$ not depending on t) with $\rho_n^t \geq C\log n/n$ and  $\rho_{n}^{t} \rightarrow 0$ for all $t \in \{1,\dots,T\}$. Then, the $\hat{k}_{K T}\left(\AnT\right)$ order estimator does not underestimate $k_{0}$, eventually almost surely when $n \rightarrow \infty$.
\end{prop} 

 \begin{proof} 
 
We start with case of the multi-layer SBM. Following the approach of \cite{CL20} in the proof of Proposition 6, to prove that $\hat{k}_{\mathbf{KT}}\left(\AnT\right)$ does not underestimate $k_{0}$ it is enough to show that for all $k<k_{0}$
\begin{equation*}
   \liminf _{n \rightarrow \infty} \frac{1}{\underset{1\leq t \leq T}{\min} \rho_{n}^{t} n^{2}}  \log \frac{\sup _{(\pi, \Svec) \in \Theta^{k_{0},T}_\ML} \mathbb{P}_{\pi,\boldsymbol \rho_{n} \Svec}\left(\AnT\right)}{\sup _{(\pi, \Svec) \in \Theta^{k,T}_\ML} \mathbb{P}_{\pi, \boldsymbol\rho_{n} \Svec}\left(\AnT\right)} > 0.
\end{equation*} 
We start with $k=k_0-1$. Using that the maximum likelihood estimators, in the complete model, for $\pi_{a}$ and $S_{a, b}^{t}$ are given by $n_{a}/n$ and $o_{a, b}^{t}/n_{a, b}$ respectively, and using $\mathbf{z}_{n,k}^{\star}$ defined in \eqref{def_z_etoile} we obtain that
\begin{align}  \label{avant_div_s_ann}
& \log  \frac{\sup _{(\pi, \Svec) \in \Theta^{k_{0},T}_\ML} \mathbb{P}_{\pi, \boldsymbol \rho_{n} \Svec}\left(\AnT\right)}{\sup _{(\pi, \Svec) \in \Theta^{k,T}_\ML} \mathbb{P}_{\pi,\boldsymbol \rho_{n} \Svec}\left(\AnT\right)}  \geq \mathcal{T}(n) \nonumber\\&
  + \frac{1}{2}\sum_{t=1}^{T}  \Biggl( \sum_{1 \leq a,b \leq k_{0}} \frac{n_{a}(\mathbf{Z}_{n}) n_{b}(\mathbf{Z}_{n})}{\rho_n^t} \gamma  \left( \rho_n^t \frac{\tilde{o}_{a, b}^{t}(\mathbf{Z}_{n}, \mathbf{A}_{n \times n}^{t})}{n_{a}(\mathbf{Z}_{n}) n_{b}(\mathbf{Z}_{n})}\right) \nonumber\\
&  - \sum_{1 \leq a,b \leq k} \frac{n_{a}(\mathbf{z}_{n,k}^{\star}) n_{b}(\mathbf{z}_{n,k}^{\star})}{ \rho_n^t} \gamma \left( \rho_n^t \frac{\tilde{o}_{a, b}^{t}(\mathbf{z}_{n,k}^{\star}, \mathbf{A}_{n \times n}^{t})}{n_{a}(\mathbf{z}_{n,k}^{\star}) n_{b}(\mathbf{z}_{n,k}^{\star})}\right) \Biggr) ,
\end{align}
with \begin{align*}
\mathcal{T}(n)=n\sum_{a=1}^{k_{0}} &\hat{\pi}_{a}(\mathbf{Z}_{n}) \log \hat{\pi}_{a}(\mathbf{Z}_{n}) \\ &- n \log k - n\sum_{a=1}^{k} \tilde{\pi}_{a}(\mathbf{z}_{n,k}^{\star}) \log \tilde{\pi}_{a}(\mathbf{z}_{n,k}^{\star}),
\end{align*}
and 
\begin{equation*} \gamma (u) = u \log u + (1-u) \log(1-u),\end{equation*}
and with $\hat{\pi}_{a}(\mathbf{Z}_{n}) = n_{a}(\mathbf{Z}_{n})/n$, for $1 \leq a \leq k_{0}$, and $\tilde{\pi}_{a}( \mathbf{z}_{n,k}^{\star}) = n_{a}( \mathbf{z}_{n,k}^{\star})/n$, for $1 \leq a \leq k$. It is reminded that for all  $1 \leq a,b \leq k$, the counter $\tilde{o}_{a b}$ is as defined in \eqref{o_tilde}.
For $\rho_{n} \rightarrow 0$ we have that
\begin{align*}
\gamma\left(\rho_{n} x\right)&=\rho_{n}(x \log x-x)+x \rho_{n} \log \rho_{n}+O\left(\rho_{n}^{2} x^{2}\right)\\
&=\rho_n \tau(x) + x \rho_n \log(\rho_n)+ O\left(\rho_{n}^{2} x^{2}\right),
\end{align*}
where $\tau(u)= u \log u -u$.
Therefore
\begin{align*}
&\sum_{t=1}^{T} \sum_{1 \leq a \leq b \leq k_0} \frac{n_{a}(\mathbf{Z}_{n}) n_{b}(\mathbf{Z}_{n})}{ \rho_n^t} \gamma \left( \rho_n^t \frac{\tilde{o}_{a, b}^{t}(\mathbf{Z}_{n}, \mathbf{A}_{n \times n}^{t})}{n_{a}(\mathbf{Z}_{n}) n_{b}(\mathbf{Z}_{n})}\right)
\\& =\sum_{t=1}^{T} \sum_{1 \leq a \leq b \leq k_0}  n_{a}(\mathbf{Z}_{n}) n_{b}(\mathbf{Z}_{n}) \tau\left(\frac{\tilde{o}_{a, b}^{t}(\mathbf{Z}_{n}, \mathbf{A}_{n \times n}^{t})}{n_{a}(\mathbf{Z}_{n}) n_{b}(\mathbf{Z}_{n})}\right) \\
& \qquad+ \sum_{t=1}^{T} E_{n}^{t} \log \rho_{n}^{t}+O \left( \left(\rho_n^t \frac{\tilde{o}_{a, b}^{t}(\mathbf{Z}_{n}, \mathbf{A}_{n \times n}^{t})}{n_{a}(\mathbf{Z}_{n}) n_{b}(\mathbf{Z}_{n})}\right) ^2 \right), 
\end{align*}
where $E_{n}^{t}=\sum_{1 \leq a \leq b \leq k_{0}} \tilde{o}_{a b}^{t}\left(\mathbf{Z}_{n}, \mathbf{A}_{n \times n}^{t}\right)$ is the total number of edges in the graph. Then, dividing both sides of \eqref{avant_div_s_ann} by $\underset{1\leq t \leq T}{\min} \rho_{n}^{t} n^{2}$, and summing on the right-hand side the following term (which equals 0 since the total number of edges in the graph does not depend on $\mathbf{Z}_n$)
\begin{align*}
&\frac{1}{\underset{1\leq t \leq T}{\min} \rho_{n}^{t} n^{2}} \times\\
&\frac{1}{2} \sum_{t=1}^{T} \Biggl( \sum_{1 \leq a, b \leq k_0} \tilde{o}_{ab}(\mathbf{Z}_{n}, {A}_{n \times n}^{t}) (\log (1/ \rho_n^t) + 1- 1/ \rho_n^t) \\ &- \sum_{1 \leq a, b \leq k} \tilde{o}_{ab}(\mathbf{z}_{n}^{\star}, {A}_{n \times n}^{t}) (\log (1/ \rho_n^t) + 1- 1/ \rho_n^t) \Biggr),
\end{align*}
and since $\mathcal{T}(n)/ \underset{1\leq t \leq T}{\min} \rho_{n}^{t} n^{2}$ converges almost surely to 0 because the $\hat \pi \log \hat \pi $ and $\tilde \pi \log \tilde \pi$ are bounded, 
to obtain the non-underestimation it is sufficient to prove that 
\begin{align*}
   & \liminf_{n \rightarrow \infty} \frac{1}{\underset{1\leq t \leq T}{\min} \rho_{n}^{t} n^{2}} \times \\
   &\sum_{t=1}^{T}  \Biggl(  \sum_{1 \leq a,b \leq k_{0}} \rho_n^t n_{a}(\mathbf{Z}_{n}) n_{b}(\mathbf{Z}_{n}) \tau \left( \frac{\tilde{o}_{a, b}^{t}(\mathbf{Z}_{n}, \mathbf{A}_{n \times n}^{t})}{\rho_n^t n_{a}(\mathbf{Z}_{n}) n_{b}(\mathbf{Z}_{n})}\right) \\
    & - \sum_{1 \leq a,b \leq k} \rho_n^t n_{a}(\mathbf{z}_{n,k}^{\star}) n_{b}(\mathbf{z}_{n,k}^{\star}) \tau \left(\frac{\tilde{o}_{a, b}^{t}(\mathbf{z}_{n,k}^{\star}, \mathbf{A}_{n \times n}^{t})}{ \rho_n^t n_{a}(\mathbf{z}_{n,k}^{\star}) n_{b}(\mathbf{z}_{n,k}^{\star})}\right) \Biggr) >0.
\end{align*}
And since for all $1 \leq t \leq T$ we have that $\rho_n^t \geq \underset{1\leq t \leq T}{\min} \rho_{n}^{t}  $ and  $\liminf_n \sum_t \geq \sum_t \liminf_n$, it is enough to prove that for all $1 \leq t \leq T$
\begin{align*}
   & \liminf_{n \rightarrow \infty} \sum_{1 \leq a,b \leq k_{0}} \frac{n_{a}(\mathbf{Z}_{n}) n_{b}(\mathbf{Z}_{n})}{ n^{2}}  \tau \left(\frac{\tilde{o}_{a, b}^{t}(\mathbf{Z}_{n}, \mathbf{A}_{n \times n}^{t})}{\rho_n^t n_{a}(\mathbf{Z}_{n}) n_{b}(\mathbf{Z}_{n})}\right) \\ &-  \sum_{1 \leq a,b \leq k} \frac{n_{a}(\mathbf{z}_{n,k}^{\star}) n_{b}(\mathbf{z}_{n,k}^{\star})}{ n^{2}} \tau \left( \frac{\tilde{o}_{a, b}^{t}(\mathbf{z}_{n,k}^{\star}, \mathbf{A}_{n \times n}^{t})}{\rho_n^t n_{a}(\mathbf{z}_{n,k}^{\star}) n_{b}(\mathbf{z}_{n,k}^{\star})}\right) \geq0, 
\end{align*}
and that there exists $1 \leq t \leq T$ such that the inequality is strict.
By Lemma \ref{lemme_prel_1} we have that
\begin{align} \label{en_k0_s}
    \lim_{n\rightarrow \infty} \frac{1}{2} &\sum_{1 \leq a,b \leq k_{0}} \frac{n_{a}(\mathbf{Z}_{n}) n_{b}(\mathbf{Z}_{n})}{ n^{2}} \tau \left(\frac{\tilde{o}_{a, b}^{t}(\mathbf{Z}_{n}, \mathbf{A}_{n \times n}^{t})}{\rho_n^t n_{a}(\mathbf{Z}_{n}) n_{b}(\mathbf{Z}_{n})}\right) \nonumber \\
    & \hspace{2cm}=  
    \frac{1}{2}  \sum_{1 \leq a,b \leq k_{0}} \pi_{a} \pi_{b} \tau \left(S^{0,t}_{ab}\right).
\end{align}
On the other hand, by Lemma \ref{lemme_prel_2} we have that
\begin{align} \label{en_k_s}
    \limsup_{n\rightarrow \infty} \frac{1}{2}  & \sum_{1 \leq a,b \leq k} \frac{n_{a}(\mathbf{z}_{n,k}^{\star}) n_{b}(\mathbf{z}_{n,k}^{\star})}{ n^{2}} \tau \left(\frac{\tilde{o}_{a, b}^{t}(\mathbf{z}_{n,k}^{\star}, \mathbf{A}_{n \times n}^{t})}{\rho_n^t n_{a}(\mathbf{z}_{n,k}^{\star}) n_{b}(\mathbf{z}_{n,k}^{\star})}\right)\nonumber  \\
    & \hspace{2cm} \leq
    \frac{1}{2} \sum_{1 \leq a,b \leq k} \pi_{a}^{\star} \pi_{b}^{\star} \tau \left(S^{\star,t}_{ab}\right),
\end{align}
for some matrix $S^{\star,t}$ in $[0,1]^{k\times k}$ and $k$-dimensional vector $\pi^{\star}$ defined by \eqref{def_S_star_pi_star}. Finally, by Lemma \ref{lemme_prel_3} we have that the difference of \eqref{en_k_s} and \eqref{en_k0_s} is lower bounded by 
\begin{align*}
    \frac{1}{2} \left(\sum_{1 \leq a,b \leq k_{0}} \pi_{a} \pi_{b} \tau \left(S^{0,t}_{ab}\right) -  \sum_{1 \leq a,b \leq k} \pi_{a}^{\star} \pi_{b}^{\star} \tau \left(S^{\star,t}_{ab}\right) \right) \geq0
\end{align*}
with a strict inequality if $S^{0,t}$ doesn't have two identical columns (there exists such a $t$ by definition of the model). Then we have proven the proposition for $k=k_0-1$. To conclude the proof for the multi-layer SBM, let $k<k_0-1$ and write
\begin{align}\label{k<k_0-1 sparse}
   & \log \frac{\sup _{(\pi, \Svec) \in \Theta^{k_{0},T}_\ML} \mathbb{P}_{\pi,\boldsymbol \rho_n\Svec}\left(\AnT\right)}{\sup _{(\pi, \Svec) \in \Theta^{k,T}_\ML} \mathbb{P}_{\pi,\boldsymbol \rho_n \Svec}\left(\AnT\right)} \nonumber \\
   & \hspace{2cm}= \log  \frac{\sup _{(\pi, \Svec) \in \Theta^{k_{0},T}_\ML} \mathbb{P}_{\pi,\boldsymbol \rho_n \Svec}\left(\AnT\right)}{\sup _{(\pi, \Svec) \in \Theta^{k_0-1,T}_\ML} \mathbb{P}_{\pi, \boldsymbol \rho_n\Svec}\left(\AnT\right)}  \\ \nonumber
   & \hspace{2cm}+ \log \frac{\sup _{(\pi, \Svec) \in \Theta^{k_{0}-1,T}_\ML} \mathbb{P}_{\pi,\boldsymbol \rho_n \Svec}\left(\AnT\right)}{\sup _{(\pi, \Svec) \in \Theta^{k,T}_\ML} \mathbb{P}_{\pi,\boldsymbol \rho_n \Svec}\left(\AnT\right)}.
\end{align}
The first term can be handled as before. The second term is non-negative because the maximum likelihood function is an increasing function of the dimension of the model and $k < k_0-1$. Then we have that \eqref{k<k_0-1 sparse} is strictly positive. This conclude the proof of the proposition for the multi-layer SBM.

\noindent For the dynamic SBM, the proof is quite similar. We just need to note that the maximum likelihood estimator, in the complete model, for $\pi_{ab}$, is given by $c_{ab}/n_a$, with $n_a=\sum_{t=1}^{T}n_a^t$, and for the matrices $S_{a b}^{t}$ for $a\neq b$, and $S_{aa}$, the maximum likelihood estimators are given 
 by $o_{a b}^{t}/n_{a b}^t$ and $ \sum_{t=1}^T o_{aa}^t/ \sum_{t=1}^T n_{aa}^t$ respectively. Then with the help of Lemma \ref{lemme_prel_1}, Lemma \ref{lemme_prel_2_dyn}, and Lemma \ref{lemme_prel_3_dyn} we obtain the result. 

\end{proof}

\section{Simulations} \label{simu}

The calculation of the penalized Krichevsky-Trofimov estimator is computationally infeasible due to the summation over all possible label configurations appearing in the likelihood. This necessitates an alternative computational approach to verify results. Consequently, we employ the Variational Bayes Expectation-Maximization (VBEM) algorithm as proposed in \cite{VBEM}, as this approach approximates the value of the Krichevsky-Trofimov estimator, up to the additional penalization term. This approach is implemented within the framework of the MLSBM; however, a VB implementation for the dynamic SBM has not yet been developed. In the MLSBM implementation, the computational cost of each VB update step follows the method of \citep{VBEM}, but the total runtime depends on the number of iterations required for convergence.
All the experiments presented in this article are reproducible, and the code used to conduct them is available at the following link: \url{https://github.com/arts-lucie/consistent-KT-estimator-in-the-MLSBM}. 

As in \cite{CL20}, we compare the performance of our estimator on simulated data with 3 other methods implemented in the R package \textit{randnet} \citep{randnet}. The first method is based on the penalized maximum likelihood (PML) and selects the number of communities with the function LRBIC, based on an asymptotic likelihood ratio, using the approach from \cite{wang_bickel}. The theoretical analysis of this method is conducted under maximum likelihood and variational Expectation–maximization algorithm, but as suggested in the paper, spectral clustering is used for community detection before fitting the maximum likelihood. The second one estimates the number of communities using the spectral properties of the network's Bethe-Hessian matrix \citep[BHMC,][]{BHMC}. The last one selects block models through network cross-validation \citep[NCV,][]{NCV}. Note that these 3 methods apply on single graphs, while ours processes a collection of graphs. Nonetheless our estimator generalizes the one in \citep{CL20} whose performances have already been shown equivalent and sometimes superior in the case of a single graph. In all the experiments, we fix the maximum value $K_{\text{max}}$ at 15.

We simulate $n_0=100$ multi-layer SBMs with $T=5$ layers and $k_0=6$ communities with uniform proportions. For each layer the connection probability matrix $P^{0,t}$ is defined as shown in Figure \ref{fig:matrice}, where $u_{1}$, $u_{2}$, $u_{3}$ and $u_{4}$ are independently and identically distributed according to a uniform distribution $\mathcal{U}(0.6,1)$. This represents a mix of communities and disassortative behaviour. 
We have tried different values of $\epsilon$ for the penalty calculation, ranging from $ 10^{-8}$ to $0.1$, in the various proposed experiments, and no significant differences in the results were observed.
 
\begin{figure}[!t]
    \centering
    \[
    \left[
    \begin{array}{ccc ccc}
    u_{1} & 0.4 & 0.4 & 0.2 & 0.2 & 0.2 \\
    0.4 & u_{2} & 0.4 & 0.2 & 0.2 & 0.2 \\
    0.4 & 0.4 & u_{3} & 0.2 & 0.2 & 0.2 \\
    0.2 & 0.2 & 0.2 & 0.4 & u_{4} & u_{4} \\
    0.2 & 0.2 & 0.2 & u_{4} & 0.4 & u_{4} \\
    0.2 & 0.2 & 0.2 & u_{4} & u_{4} & 0.4
    \end{array}
    \right]
    \]
    \caption{Matrix $P^{0,t}$ where $u_{1}$, $u_{2}$, $u_{3}$, and $u_{4}$ $\overset{\text{i.i.d.}}{\sim}$ $\mathcal{U}(0.6,1)$.}
    \label{fig:matrice}
\end{figure}

For others methods than the penalized KT estimator, the estimator is applied to each layer, and the number of correct results (selection of the correct number of clusters) is counted. To calculate the accuracy, the number of correct results is divided by the total number of observed graphs (i.e. the number of simulated graphs $n_s$ for the penalized KT estimator and  $n_s \times T = 500$ for the other estimators). While our approach outputs a unique $0/1$ accuracy per collection (the number of clusters is correctly estimated or not), the other 3 methods may show accuracies in $[0,1]$ per collection (the number of clusters may be correctly estimated for some of the $5$ graphs in each collection), this slightly favoring the latter.

\begin{figure}[!t]
    \centering
     \includegraphics[  height=5cm]{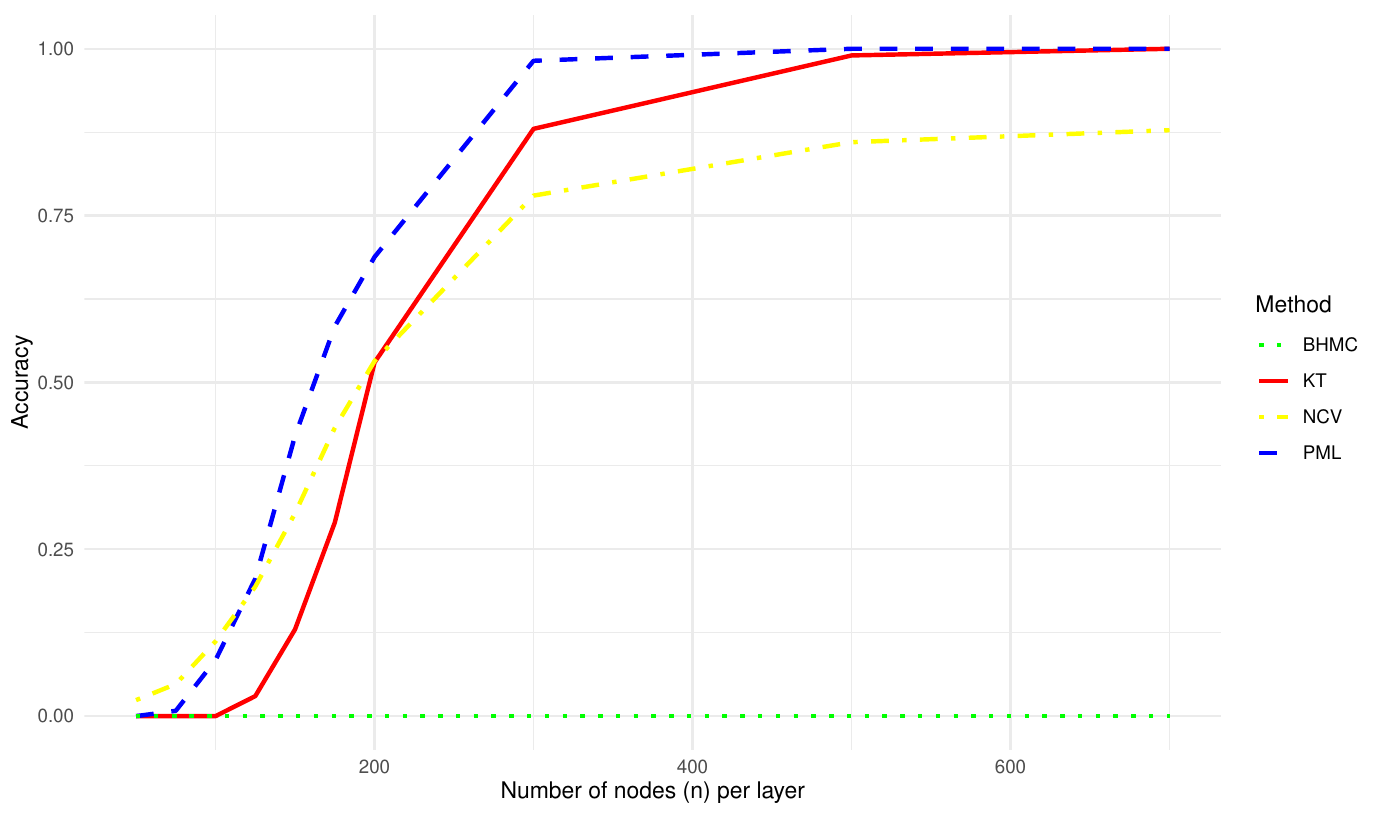} 
    \caption{Comparaison of the accuracy of different methods: the penalized Krichevsky-Trofimov estimator (KT), the penalized maximum likelihood (PML), the Bethe-Hessian matrix with moment correction (BHMC) and the network cross-validation method (NCV).}
    \label{fig:figure_comparaison_results}
\end{figure}

 Figure \ref{fig:figure_comparaison_results} presents the accuracy calculated over 100 collections of graphes of each estimator as a function of the number of nodes per layer for $n \in \{50,75,100,125,150,175,200,300,500,700\}$. In this scenario, the penalized KT estimator performs significantly better than the BHMC estimator, which fails to determine the correct number of clusters. Although it is initially slightly less accurate than the PML estimator, as the number of nodes increases both estimators achieve asymptotically perfect accuracy, making their residual error indistinguishable. Similarly, although it is slightly less effective than the NCV estimator when the number of nodes is small, it quickly surpasses it and remains superior.

To further assess our estimator's accuracy, we decide to test it in a sparse regime context. Data were simulated from a model with 4 layers and 3 communities, comprising 300 nodes per layers. The connection probability matrix $ P_0^{1:T} $ was defined as $ P_0^{1:T} = \rho S_0^{1:T} $, where $ S_0^t $ is a matrix such that $S_0^t = Id + \bf{1} + \epsilon_t $ where $\epsilon_t$ are i.i.d such that $\epsilon_t \sim \mathcal{U}([-0.1,0.1])$. The parameter $ \rho $ varied from 0.05 to 0.45, enabling an analysis of the estimator's accuracy over 100 trials at different levels of connection density. Table \ref{tab:sparse} presents a summary of these results as well as the results for the other estimators discussed previously.

\begin{table}[!t]
    \caption{Accuracy of the penalized KT estimator, the PML estimator, the BHMC estimator and the NCV estimator across varying values of $\rho$.}

    \centering
    \begin{tabular}{|c|c|c|c|c|c|c|c|c|c|c|}
        \hline
        $\rho$ & 0.05 & 0.10 & 0.15 & 0.20 & 0.25 & 0.30 & 0.35 & 0.40 & 0.45  \\
        \hline 
        KT & 0.00 & 0.06 &0.49 & 0.71 & 0.84 & 0.95 & \textbf{1.00} & \textbf{1.00} & \textbf{1.00}  \\
        \hline
        PML  & 0.00 &\textbf{0.83} & \textbf{1.00} &\textbf{1.00} &\textbf{1.00} &\textbf{1.00} &\textbf{1.00} &\textbf{1.00} &\textbf{1.00} \\
        \hline
        BHMC & 0.00 & \textbf{0.83} &  \textbf{1.00} &\textbf{1.00} &\textbf{1.00} &\textbf{1.00} &\textbf{1.00} &\textbf{1.00} &\textbf{1.00}  \\
        \hline
        NCV & 0.01 & 0.78 & 0.98 & 0.99 &0.99& 0.99 & 0.98 & 0.96 & 0.92 \\
        \hline
    \end{tabular}
    \label{tab:sparse}
\end{table} 

This table indicates that the accuracy of the penalized KT estimator, as well as that of the other estimators, improves significantly as $ \rho $ increases. We also observe the detection threshold, the limit below which groups cannot be estimated due to a lack of edges. The difference in efficiency between the penalized KT estimator and the PML estimator for small $\rho$ is probably due to the calculation of accuracy and the fact that the KT penalty might be too large in this case. However, we note that even with a not-too-large sparsity parameter, the penalized KT estimator already performs quite well in this context.

We also attempt to observe the rate of convergence of the penalized KT estimator. To this end, we plot the accuracy of the estimator for different numbers of layers $ T \in \{1,4,9,16\}$ using multi-layer SBM graphs in a community scenario. Each layer's connection probability matrix is defined such that each diagonal element is independently drawn from a uniform distribution $\mathcal{U}(0.7,1)$, and the anti-diagonal element is drawn from a uniform distribution $\mathcal{U}(0,0.1)$. These graphs have a number $ n $ of nodes per layer such that, in Figure \ref{fig:fig_T}, $ n \times T $ always equals the same $ N_{\text{tot}} $ (ensuring a constant total number of nodes), and in Figure \ref{fig:fig_sqrt_T}, $ n^2 \times T $ is always equal to the same $ N_{\text{tot}} $ (ensuring a constant maximum number of interactions). 

\begin{figure}[!t]
    \centering
    \begin{subfigure}[t]{0.45\textwidth}
        \centering
        \includegraphics[width=\textwidth, height=5cm]{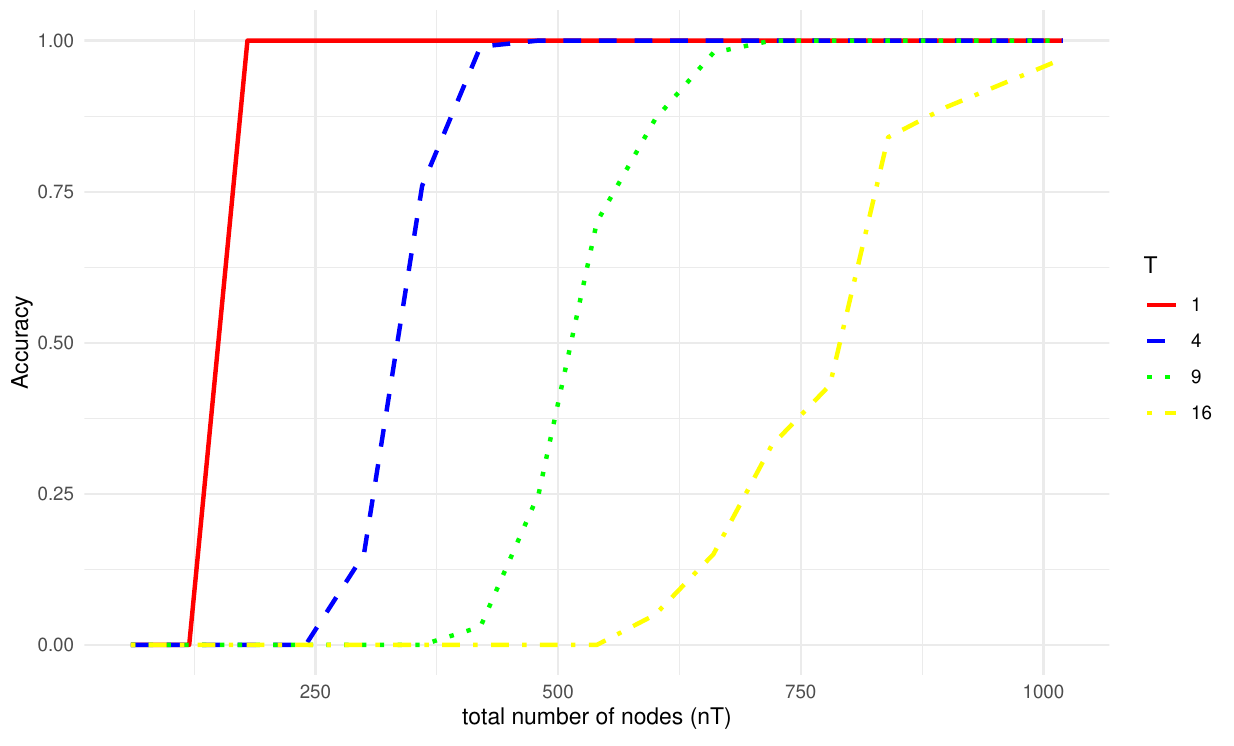}
        \caption{constant total number of nodes}
        \label{fig:fig_T}
    \end{subfigure}
    \hfill
    \begin{subfigure}[t]{0.45\textwidth}
        \centering
        \includegraphics[width=\textwidth, height=5cm]{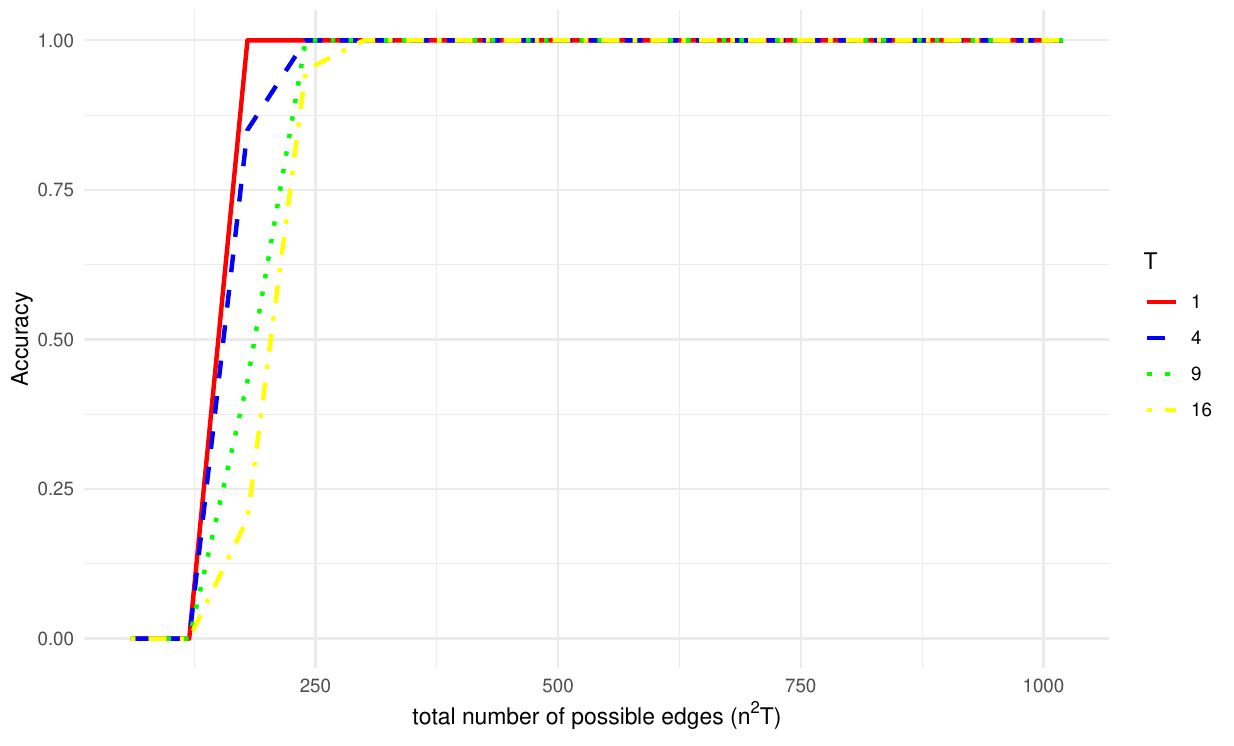} 
        \caption{constant maximum number of interactions}
        \label{fig:fig_sqrt_T}
    \end{subfigure}
    \caption{Illustration of convergence rates of the penalized KT-estimator for different number of layers.}
    \label{fig:main}
\end{figure}

In the first case (constant total number of nodes), since the curves are not superimposed, this suggests that the convergence speed of our estimator is not in $ 1 / \sqrt{nT} $ (i.e., $ 1 / \sqrt{\text{total number of nodes}} $). In contrast, in the second case (constant maximum number of interactions), the curves are almost superimposed, which suggests a convergence rate of $ 1 / (n \sqrt{T}) $ (or $ 1 / \sqrt{\text{maximum number of interactions}} $).

\begin{figure}[!t]
        \centering
        \includegraphics[height=5cm]{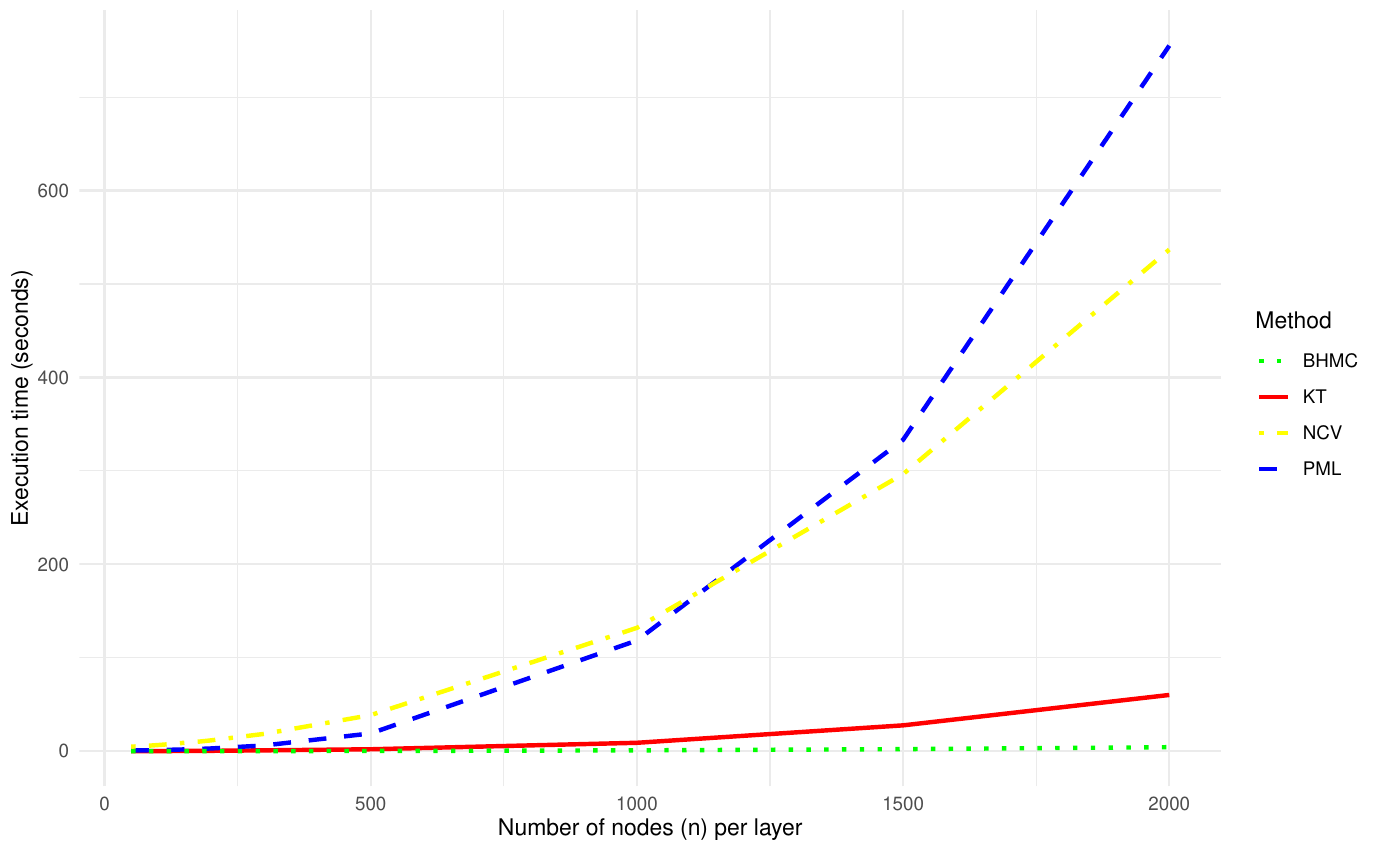} 
        \caption{computation time of the different methods}
        \label{fig:figure_comparaison_temps}
 \end{figure}  
    
As a final experiment, we aim to observe the computation time. It is presented in Figure \ref{fig:figure_comparaison_temps} as a function of the number of nodes $n \in \{50,75,100,125,150,175,200,300,500,1000,1500,2000    \} $ per layer, averaged over 10 simulated graphs which were generated as in the first scenario using the $P^{0,t}$ from Figure \ref{fig:matrice}. This experiment was performed in R (version 4.4.1) on a MacBook Air (2022) with an Apple M2 chip running macOS Sequoia 15.2. We observe that the penalized KT estimator runs very quickly, much faster than the PML and NCV estimators, and is surpassed only by the BHMC. For example, evaluating a $T=5$ layers multi-layer graphs with $n=2000$ nodes per layer takes one minute with the penalized KT estimator, while the BHMC, NCV, and PML estimators take 4 seconds, 9 minutes, and 12 minutes and 30 seconds respectively over the same collection. Thus, the penalized KT estimator is of course slower than a spectral method but it is very competitive with the other two. By interpreting these results in parallel with those of the first experiment shown in Figure \ref{fig:figure_comparaison_results}, the penalized KT estimator emerges as the fastest among the estimators that successfully estimate the correct number of clusters in this scenario.

\appendices  
\renewcommand{\thesubsection}{\thesection.\arabic{subsection}}

\section{Proofs of technical results for the consistency theorem}

\subsection{Proof of the uniform bound proposition (Proposition \ref{propfond})} \label{preuve_propfond}
The proof is a direct generalization of the proof of Proposition 1 in \cite{CL20} for the case with multiple layers or time points. 

The first inequality is a direct consequence of the definition of the $\mathbf{KT}$ distribution. We start by proving the second inequality of the proposition for a multi-layer SBM.  We first note that we have 
\begin{equation*} 
\mathbb{P}_{\pi, \Pvec}(\mathbf{a}_{n \times n} ^{1:T} \mid \mathbf{z}_{n})=\prod_{t=1}^{T} \prod_{1 \leq a \leq b \leq k} \left(P_{a b}^{t}\right) ^ {o_{a b}^{t}}\left(1-{P_{a b}^{t}}\right)^{n_{a b}-o_{a b}^{t}},
\end{equation*}
and that the maximum likelihood estimator, in the complete model, for $P_{a b}^{t}$ is given by $o_{a b}^{t}/n_{a b}$. 
Then with the same calculations as in \cite{CL20}, we have that 
\begin{equation*}
\frac{\mathbb{P}_{\pi, \Pvec}\left(\mathbf{z}_{n}, \anT\right)}{\mathbf{KT}_{k}^{T}\left(\mathbf{z}_{n}, \anT\right)} \leq e^{C\left(\mathbf{z}_{n}, \anT\right)}, 
\end{equation*}
where
\begin{align*}
C\left(\mathbf{z}_{n}, \anT\right)= & \log \left(\frac{\Gamma\left(\frac{1}{2}\right) \Gamma\left(n+\frac{k}{2}\right)}{\Gamma\left(\frac{k}{2}\right) \Gamma\left(n+\frac{1}{2}\right)}\right) \\ &+\sum_{t=1}^{T}\sum_{1 \leq a \leq b \leq k} \log \left(\frac{\Gamma\left(\frac{1}{2}\right) \Gamma\left(n_{a b}+1\right)}{\Gamma\left(n_{a b}+\frac{1}{2}\right)}\right).
\end{align*}
Then, by using Stirling's formula for the $\Gamma$ function, we write
\begin{equation*}
C\left(\mathbf{z}_{n}, \anT\right) \leq \frac{Tk(k+1)+k-1}{2} \log n+c_{k,T}, 
\end{equation*}
where
$
c_{k,T}=Tk(k+1)+1
$.
This concludes the proof of the proposition for a multi-layer SBM. 

We now prove it for a dynamic SBM. For the second inequality in \eqref{eqprop1_dyn} we have that, given $\mathbf{z}_n^1$, for $(\Pi, \Pvec) \in \Theta^{k,T}_\dyn$
\begin{equation}
\mathbb{P}_{\Pi}(\znT|\mathbf{z}_n^1)= \prod_{1 \leq a, b \leq k} \pi_{ab}^{c_{ab}},
\label{Pz_dyn}
\end{equation}
and
\begin{equation}
\mathbb{P}_{\Pi, \Pvec}(\anT \mid \znT)=\prod_{t=1}^{T} \prod_{1 \leq a \leq b \leq k} \left(P_{a b}^{t}\right) ^ {o_{a b}^{t}}\left(1-{P_{a b}^{t}}\right)^{n_{a b}^t-o_{a b}^{t}}.
\label{Paz_dyn}
\end{equation}
Using that the maximum likelihood estimator, in the complete model, for the transition probability $\pi_{ab}$, is given by $c_{ab}/n_a$, with $n_a=\sum_{t=1}^{T}n_a^t$, and using that for the matrices $P_{a b}^{t}$ for $a\neq b$, and $P_{aa}$, the maximum likelihood estimators are given by $o_{a b}^{t}/n_{a b}^t$ and $o_{aa}/n_{aa}$ respectively, where $o_{aa}= \sum_{t=1}^T o_{aa}^t$ and $n_{aa}= \sum_{t=1}^T n_{aa}^t$, we can bound from above \eqref{Pz_dyn} and \eqref{Paz_dyn} by
\begin{equation*}
\mathbb{P}_{\Pi}(\mathbf{z}_{n}^{2:T} |\mathbf{z}_{n}^1) \leq \prod_{1 \leq a, b \leq k} \left(\frac{ c_{ab}}{ n_a} \right)^{c_{ab}}, 
\end{equation*}
and
\begin{align*}
 &\mathbb{P}_{\Pi, \Pvec}(\anT \mid\znT) \\ 
 &\leq \prod_{t=1}^{T} \prod_{1 \leq a < b \leq k} \left({\frac{o_{a b}^{t}}{n_{a b}^t}} \right)^ {o_{a b}^{t}}\left(1-{\frac{o_{a b}^{t}}{n_{a b}^t}}\right)^{n_{a b}^t-o_{a b}^{t}}  \\
 & \hspace{2cm}\times  \prod_{1 \leq a \leq k} \left({\frac{o_{a a}}{n_{a a}}} \right)^ {o_{a a}}\left(1-{\frac{o_{a a}}{n_{a a}}}\right)^{n_{a a}-o_{a a}}.
\end{align*}
Then by some calculations we obtain 
\begin{equation*}
\frac{\mathbb{P}_{\Pi,\Pvec }\left(\anT\mid \znT\right)}{\mathbf{KT}_{k}^{T}\left(\anT \mid \znT \right)} \leq e^{C\left(\znT, \anT\right)}, 
\end{equation*}
where
\begin{align*}
C\left(\znT, \anT\right)=& \sum_{a=1}^k \log \left(\frac{\Gamma\left(\frac{1}{2}\right) \Gamma\left(n_a+\frac{k}{2}\right)}{\Gamma\left(\frac{k}{2}\right) \Gamma\left(n_a+\frac{1}{2}\right)}\right) \\
& + \sum_{1 \leq a \leq k} \log \left(\frac{\Gamma\left(\frac{1}{2}\right) \Gamma\left(n_{a a}+1\right)}{\Gamma\left(n_{a a}+\frac{1}{2}\right)}\right) \nonumber \\
& + \sum_{t=1}^{T}\sum_{1 \leq a < b \leq k} \log \left(\frac{\Gamma\left(\frac{1}{2}\right) \Gamma\left(n_{a b}^t+1\right)}{\Gamma\left(n_{a b}^t+\frac{1}{2}\right)}\right).
\end{align*}
Then we bound from above $C$ by
\begin{align*}
&C\left(\znT, \anT\right) \\
&\leq  \frac{k}{2} \log(n^2T)+\frac{k(k-1)}{2}  \log (nT) +\frac{Tk(k-1)}{2} \log n+c_{k,T}, 
\end{align*}
where
$
c_{k,T}=\frac{k}{3T}[k(k-1)+2]+Tk(k-1)+2k
$, and this concludes the proof.

\subsection{Proof of a technical lemma used in the proof of the non-underestimation}

\begin{lemme} \label{lemme_technique}
    Let $f$ be a continuous real-valued function over the set of $k \times k$ matrices with non-negative entries, and let $n(\mathbf Z)/n \rightarrow \pi$ almost surely as $n \rightarrow \infty$ at a rate $a_n \downarrow 0$, and where $\pi$ is a vector with positive entries. We have that almost surely \begin{equation*}
    \underset{n \rightarrow \infty}{\limsup} \sup_{\substack {Q_n : \|Q_n\|_1=1\\Q_n^{\intercal}{\bf{1}}_{k} = n(\mathbf Z)/n}} f(Q_n) \leq \sup_{\substack {R : \|R\|_1=1\\R^{\intercal}{\bf{1}}_{k} = \pi}} f(R). \end{equation*}
\end{lemme}

\begin{proof}
    For any constant $c>0$ and any integer $n$, consider the following sets of $k\times k$ matrices
    \begin{equation*}A_n=\{Q_n : \|Q_n\|_1=1 \text{ and } Q_n^{\intercal}{\bf{1}}_{k} = n(\mathbf Z)/n\},
    \end{equation*} 
    \begin{equation*}B_n(c)=\{Q_n : \|Q_n\|_1=1 \text{ and } \| Q_n^{\intercal}{\bf{1}}_{k} - \pi \|_{\infty} \leq c \times a_n \},\end{equation*}
    and \begin{equation*}B=\{R : \|R\|_1=1 \text{ and } R^{\intercal}{\bf{1}}_{k} =\pi\}.
    \end{equation*} 
    Note that there exists $c>0$ such that, for all $n$, we have $A_n \subseteq B_n(c)$ and $B_n(c) \underset{n \rightarrow \infty}{ \rightarrow} B$ for the Hausdorff distance. So we have that 
    \begin{equation*}\underset{Q_n\in A_n}{\sup}f(Q_n) \leq \underset{Q_n \in B_n(c)}{\sup}f(Q_n).\end{equation*}
     Since $f$ is continuous and $B_n(c)$ is a compact set for all $n$, the supremum is attained. Let \begin{equation*}Q_n^{\max}= \underset{Q_n \in B_n(c)}{\arg \max}f(Q_n).\end{equation*} Since we have that $B_n(c)$ is a compact for all $n$, and for all $n$ we have that $B_{n+1}(c) \subseteq B_n(c)$, we can extract a sub-sequence $( Q_{n_j}^{\max} )_j$ of $\left( Q_{n}^{\max} \right)_n$ which, by compactness of $B_1(c)$ converges to a matrix $R^\star$ in $B_1(c)$. Then we want to prove that $R^\star$ is in $B$. Let us proceed with a proof by contradiction. If $R^\star$ is not in $B$, we can find $\epsilon >0$ and $\eta >0$ such that $d_{\mathcal{H}}(\mathcal{B}(R^\star,\epsilon),B) > \eta$, where $d_{\mathcal{H}}$ is the Hausdorff distance and $\mathcal{B}(R^\star,\epsilon)$ is the ball for the $\| \cdot \|_1$ norm over matrices centered at $R^\star$ with radius $\epsilon$. Since $( Q_{n_j}^{\max} )_j \rightarrow R^\star$ there exists $n_0$ such that for all $j>n_0$ we have $( Q_{n_j}^{\max} )_j \in \mathcal{B}(R^\star,\epsilon)$. So, for all $j>n_0$
    \begin{align*}
        d_{\mathcal{H}}(B_{n_j}(c),B)& = \max \{\underset{M \in B_{n_j}(c)}{\sup}d(M,B),\underset{N \in B}{\sup}d(B_{n_j},N)\} \\
        & \geq \underset{M \in B_{n_j}(c)}{\sup}d(M,B) \\
        & \geq d(Q_{n_j}^{\max},B) \\
        & \geq \eta >0.
    \end{align*}
    However $d_{\mathcal{H}}(B_{n_j},B) \rightarrow 0$ which gives a contradiction. So we must have that $R^\star$ is in $B$. Then by continuity of $f$ we have that $f( Q_{n_j}^{\max})$ converges towards $f(R^\star)$. Moreover we have that
    \begin{align*}
    \underset{n \rightarrow \infty}{\limsup} \sup_{Q_n \in B_n(c)} f(Q_n) &= \underset{k \rightarrow \infty}{\lim} \sup_{n\geq k}\sup_{Q_n \in B_n(c)} f(Q_n) \nonumber \\
    & = \underset{k \rightarrow \infty}{\lim} \sup_{Q_k \in B_k(c)} f(Q_k) \\
    & = \underset{k \rightarrow \infty}{\lim}  f(Q_k^{\max}) \nonumber \\
   &=  \underset{j \rightarrow \infty}{\lim}  f(Q_{n_j}^{\max}) = f(R^\star), \nonumber
    \end{align*}
 where the second equality comes from the fact that $B_{n+1}(c) \subseteq B_n(c)$ for all $n$ and the fourth follows from that the limit exists and thus equals any subsequence limit. So we have that
    \begin{align*}
        \underset{n \rightarrow \infty}{\limsup} \sup_{\substack {Q_n : \|Q_n\|_1=1\\Q_n^{\intercal}{\bf{1}}_{k} = n(\mathbf Z)/n}} f(Q_n) & \leq  \underset{n \rightarrow \infty}{\limsup} \sup_{Q_n \in B_n(c)} f(Q_n) \\
        &= f(R^\star)\\
        & \leq \sup_{\substack {R : \|R\|_1=1\\R^{\intercal}{\bf{1}}_{k} = \pi}} f(R), 
    \end{align*}
    and this proves the lemma.
\end{proof}

\section{Proofs of the non-overestimation part of the consistency theorem}

\subsection{Proof of Lemma \ref{lemme_exp}} \label{preuve_lemme_exp}
    We start the proof with the multi-layer SBM.
For $k>k_{0} $ we have
\begin{align}
    &\mathbb{P}_{\pi^{0}, \Pvec^{0}}  (\hat{k}_{K T}  =k) = \mathbb{E}_{\pi^{0}, \Pvec^{0}}\left[\mathds{1}\{ \hat{k}_{K T} =k\} \right] \nonumber \\
    & = \sum_{\anT} \mathbb{P}_{\pi^{0}, \Pvec^{0}}(\anT) \mathds{1}\Biggl\{\underset{1 \leq k^{\prime} \leq n}{\arg \max } \biggl\{ \log \mathbf{KT}_{k{\prime}}^{T}(\anT) \nonumber \\
    &\hspace{5cm} - \pen_\ML(k^{\prime},n,T) \biggr\} =k\Biggr\} \nonumber\\
    & \leq \sum_{\anT} \mathbb{P}_{\pi^{0}, \Pvec^{0}}(\anT) \mathds{1} \biggl\{ \log \mathbf{KT}_{k}^{T}(\anT) - \pen_\ML(k,n,T) \nonumber \\
    & \hspace{3cm}\geq \log \mathbf{KT}_{k_{0}}^{T}(\anT) - \pen_\ML(k_{0},n,T)  \biggr\} \nonumber \\
    & = \sum_{\anT} \mathbb{P}_{\pi^{0}, \Pvec^{0}}(\anT) \mathds{1} \biggl\{ \mathbf{KT}_{k_{0}}^{T}(\anT) \leq \mathbf{KT}_{k}^{T}(\anT) \times \nonumber \\
    & \hspace{3cm} \exp[\pen_\ML(k_{0},n,T) - \pen_\ML(k,n,T)]\biggr\}.
    \label{ineq_somme_P}
\end{align}
By using the inequality \eqref{eqprop1} of Proposition \ref{propfond},
\begin{align*}
&\log \mathbb{P}_{\pi^{0}, \Pvec^{0}}(\anT)  \leq \log \sup_{(\pi,\Pvec) \in \Theta^{k_{0},T}}\mathbb{P}_{\pi, \Pvec}(\anT) \\
& \leq \log \mathbf{KT}_{k_{0}}^{T}(\anT) + \frac{Tk_{0}(k_{0}+1)+k_0-1}{2} \log n +c_{k_{0},T},
\end{align*}
and therefore, 
\begin{equation}
    \mathbb{P}_{\pi^{0}, \Pvec^{0}}(\anT) \leq \mathbf{KT}_{k_{0}}^{T}(\anT) n^{\frac{Tk_{0}(k_{0}+1)+k_0-1}{2} } e^{c_{k_{0},T}}.
    \label{ineq_P}
\end{equation}
Applying \eqref{ineq_P} in \eqref{ineq_somme_P}, and denoting $d_{k_{0}, k, n,T} = \pen_\ML(k_{0},n,T) - \pen_\ML(k,n,T)$, we have that 
\begin{align*}
&\mathbb{P}_{\pi^{0}, \Pvec^{0}}(\hat{k}_{K T}  =k) \\ 
& \leq \sum_{\anT} \mathbf{KT}_{k_{0}}^{T}(\anT) n^{\frac{Tk_{0}(k_{0}+1)+k_0-1}{2} } e^{c_{k_{0},T}} \times \\
& \hspace{2cm} \mathds{1} \left\{ \mathbf{KT}_{k_{0}}^{T}(\anT) \leq \mathbf{KT}_{k}^{T}(\anT) \exp[d_{k_{0}, k, n,T}]\right\}\\
& \leq \sum_{\anT} \mathbf{KT}_{k}^{T}(\anT) e^{d_{k_{0}, k, n,T}} n^{\frac{Tk_{0}(k_{0}+1)+k_0-1}{2} } e^{c_{k_{0},T}}\\
&= \exp \left\{ \frac{Tk_{0}(k_{0}+1)+k_0-1}{2} \log n + c_{k_{0},T} +d_{k_{0}, k, n,T}\right\},
\end{align*}
where the last equality follows from the fact that $\mathbf{KT}_{k}(\cdot)$ is a probability distribution. This concludes the proof of Lemma \ref{lemme_exp} for a multi-layer SBM. The proof is the same for a dynamic SBM, using inequality \eqref{eqprop1_dyn} in Proposition \ref{propfond}.

\subsection{Proof of Lemma \ref{lemme_pas_k,log}} \label{preuve_lemme_pas_k,log}

We prove the result in the case of a multi-layer SBM. First observe that
\begin{align}\label{premier_truc}
    &\mathbb{P}_{\pi^{0}, \Pvec^{0}}\left(\hat{k}_{\mathbf{KT}}\left(\anT\right)  \in\left(k_{0}, \log n\right]\right) \nonumber\\
    & \hspace{3cm}=  \sum_{k=k_{0}+1}^{\log n} \mathbb{P}_{\pi^{0}, \Pvec^{0}}\left(\hat{k}_{\mathbf{KT}}\left(\anT\right)=k\right). 
\end{align}
Using Lemma \ref{lemme_exp} and the fact that $k \mapsto \pen_\ML(k,n,T)$ is an increasing function for all $n$ and $T$, we can bound the sum in the right-hand side by
\begin{align} \label{deuxieme_truc}
&\sum_{k=k_{0}+1}^{\log n}  \exp \left\{\frac{Tk_{0}(k_{0}+1)+k_0-1}{2} \log n  +c_{k_{0},T}+d_{k_{0}, k, n,T}\right\}  \nonumber\\
& \leq e^{c_{k_{0},T}}( \log n )\exp \left\{\frac{Tk_{0}(k_{0}+1)+k_0-1}{2} \log n+d_{k_{0}, k_{0}+1, n,T}\right\},
\end{align}
where we recall that $d_{k_{0}, k, n,T}=\pen_\ML\left(k_{0}, n,T\right)-\pen_\ML(k, n,T)$ and $c_{k_{0},T}=Tk_{0}\left(k_{0}+1\right)+1$.  By definition
\begin{equation*}
\pen_\ML(k, n,T)=\sum_{i=1}^{k-1}\left[\frac{Ti(i+1)+i-1}{2} +1+\epsilon\right] \log n,
\end{equation*}
and therefore
\begin{align*}
&\frac{Tk_{0}(k_{0}+1)+k_0-1}{2} \log n+  d_{k_{0}, k_{0}+1, n,T} \nonumber \\
 & =\left(\frac{Tk_{0}(k_{0}+1)+k_0-1}{2}-\frac{Tk_{0}(k_{0}+1)+k_0-1}{2}-1-\epsilon\right) \log n \nonumber \\
& =-(1+\epsilon ) \log n.
\end{align*}
We obtain by using this in \eqref{deuxieme_truc} to bound \eqref{premier_truc} that
\begin{equation*}
\sum_{n=1}^{\infty} \mathbb{P}_{\pi^{0}, \Pvec^{0}}\left(\hat{k}_{\mathbf{KT}}\left(\anT\right) \in\left(k_{0}, \log n\right]\right) \leq e^{c_{k_{0},T}} \sum_{n=1}^{\infty} \frac{\log n}{n^{1+\epsilon }}
<\infty .
\end{equation*}Now the result follows by the first Borel Cantelli lemma.
The proof for a dynamic SBM is the same using the part of the dynamic SBM in Lemma \ref{lemme_exp}. 

\subsection{Proof of Lemma \ref{lemma_pas_log,n}} \label{preuve_lemma_pas_log,n}
We write the proof in the case of a multi-layer SBM. As in the proof of Lemma \ref{lemme_pas_k,log} we write
\begin{align*}
 &\mathbb{P}_{\pi^{0},\Pvec^{0}}\left(\hat{k}_{\mathbf{KT}}\left(\anT\right) \in(\log n, n]\right) \\
&\hspace{3cm }=\sum_{k=\log n}^{n} \mathbb{P}_{\pi^{0}, \Pvec^{0}}\left(\hat{k}_{\mathbf{KT}}\left(\anT\right)=k\right),
\end{align*}
and we use Lemma \ref{lemme_exp} and the increasing property of $k \mapsto \pen_\ML(k,n,T)$ to bound the sum in the right-hand side by
\begin{align}
 &\sum_{k=\log n}^{n} \exp \left\{\frac{Tk_{0}(k_{0}+1)+k_0-1}{2} \log n+c_{k_{0},T}+d_{k_{0}, k, n,T}\right\} \nonumber \\
& \leq e^{c_{k_{0},T}} n \exp \left\{\log n\left[\frac{Tk_{0}(k_{0}+1)+k_0-1}{2}+\frac{d_{k_{0}, \log n, n,T}}{\log n}\right]\right\}.
\label{deuxi_log,n}
\end{align}
Since $\pen_\ML(k, n,T) / \log (n)$ does not depend on $n$ and increases linearly in $T$ and cubically in $k$ we have that
\begin{align*}
&\frac{Tk_{0}(k_{0}+1)+k_0-1}{2}+\frac{d_{k_{0}, \log n, n,T}}{\log n} \\ 
& =\frac{Tk_{0}(k_{0}+1)+k_0-1}{2}+\frac{\pen_\ML\left(k_{0}, n,T\right)}{\log n} \\
& \hspace{5cm}-\frac{\pen_\ML(\log n, n,T)}{\log n} \\
&<-3,
\end{align*}
for all sufficiently large $n$ and fixed $T,k_0$. Thus summing the right-hand side of \eqref{deuxi_log,n} in $n$ we obtain
\begin{align*}
&\sum_{n=1}^{\infty} e^{c_{k_{0},T}} n e^{\left\{\log n\left[\frac{Tk_{0}(k_{0}+1)+k_0-1}{2}+\frac{d_{k_{0}, \log n, n,T}}{\log n}\right]\right\}}<\infty,
\end{align*}
and the result follows from the first Borel Cantelli lemma.
The proof for a dynamic SBM is the same using the part of the dynamic SBM in Lemma \ref{lemme_exp}.

\section{Proofs of the non-underestimation part of the consistency theorem}

\subsection{Proof of the concentration lemma (Lemma \ref{lemme_prel_1})} \label{preuve_lemme_prel_1}
For any fixed $\mathbf z\in [k_0]^n$ and $\mathbf{\bar z} \in[k]^n$ we have that for all $a,b \in [k]^2$
\begin{align*}
\tilde o_{ab}(\mathbf{\bar z},\mathbf{A}_{n\times n}^{t})  &-  \rho_n^t n^2 [Q_n(\mathbf{\bar z},\mathbf{z})S^{0,t} Q_n(\mathbf{\bar z},\mathbf z)^{\intercal}]_{ab} \\
&= \sum_{1\leq i,j\leq n}(\mathbf{A}_{ij}^{t}-  P^{0,t}_{{z}_{i}{z}_{j}}) \mathds{1}\{{\bar z}_i=a,{\bar z}_j=b\}.
\end{align*}
We first assume that $a\neq b$. Observe that  given $\mathbf{Z}=\mathbf{z}$, the counters $\tilde o_{ab}(\mathbf{\bar z},\mathbf{A}_{n\times n}^{t})$ correspond to the sum of $n_a(\mathbf{\bar z})n_b(\mathbf{\bar z})$ independent Bernoulli random variables  \blue{where $\overline{z}_i=a$ and $\overline{z}_j=b$}, given by $\mathbf{A}^{t}_{ij}$, with variance given by $P^{0,t}_{{z}_i  {z}_j} (1- P^{0,t}_{{z}_i  {z}_j} )$. Using Bernstein's inequality, the fact that $P^{0,t}_{{z}_i  {z}_j} (1- P^{0,t}_{{z}_i  {z}_j} ) \leq \rho_n^t S^{0,t}_{\max} $ with $S^{0,t}_{\max}= \max_{a,b}S^{0,t}_{ab}$ and the fact that $n_a(\mathbf{\bar z})n_b(\mathbf{\bar z}) \leq n^2$
, we have that for any $\delta>0$
\begin{align*}
&\mathbb P_{\pi^0, \Pvec^{0}}\Bigl(\Bigr|\tilde o_{ab}(\mathbf{\bar z},\mathbf{A}_{n\times n}^{t}) \\
&\hspace{2cm}- \rho_n^t n^2[Q_n(\mathbf{\bar z},\mathbf z)S^{0,t} Q_n(\mathbf{\bar z},\mathbf z)^{\intercal}]_{ab}\Bigr|>\delta \mid \mathbf{Z}=\mathbf z\Bigr)   \\
&\leq 2 \exp \left( \frac{- \delta ^2}{2n^2 \rho_n^t S^{0,t}_{\max} + \frac{2}{3}\delta}   \right)\\
&\leq 2 \exp \left( \frac{- \delta ^2}{4n^2 \rho_n^t S^{0,t}_{\max} + \frac{4}{3}\delta}   \right).
\end{align*}
We can rewrite this bound, unsing $\xi = \delta / (\rho_n^t n^{2})$ 
\begin{align*}
&\mathbb P_{\pi^0, \Pvec^{0}}\Bigl(\Bigr|\frac{\tilde o_{ab}(\mathbf{\bar z},\mathbf{A}_{n\times n}^{t})}{\rho_n^t n^2} \\
& \hspace{2cm}- [Q_n(\mathbf{\bar z},\mathbf{z})S^{0,t} Q_n(\mathbf{\bar z},\mathbf z)^{\intercal}]_{ab}\Bigr|>\xi \mid \mathbf{Z}=\mathbf z\Bigr) \\
&\leq 2 \exp \left( \frac{- (\rho_n^t)^2 n^{4} \xi ^2}{4 n^2\rho_n^t S^{0,t}_{\max} + \frac{4}{3} \xi \rho_n^t n^2}  \right) \\ 
&=2 \exp \left(- \frac{\rho_n^t n^{2} \xi ^2}{4 S^{0,t}_{\max}  + \frac{4}{3}\xi} \right)  .
\end{align*}

We first focus on the case of the multi-layer SBM. Using a union bound over all $\mathbf{\bar z}\in [k]^n$ and integrating over $\mathbf{z}$ we obtain that 
\begin{align*}
&\mathbb P_{\pi^0, \Pvec^{0}}\Bigl( \sup_{\mathbf{\bar z} \in [k]^n}\, \Bigr|  \frac{\tilde o_{ab}(\mathbf{\bar z},\mathbf{A}_{n\times n}^{t})}{\rho_n^t n^2} - [Q_n(\mathbf{\bar z},\mathbf z)S^{0,t} Q_n(\mathbf{\bar z},\mathbf z)^{\intercal}]_{ab}\Bigr|>\xi\Bigr) \\
& \leq \sum_{\mathbf{\bar z} \in [k]^n} \mathbb P\Bigl(\Bigr|\frac{\tilde o_{ab}(\mathbf{\bar z},\mathbf{A}_{n\times n}^{t})}{\rho_n^t n^2} - [Q_n(\mathbf{\bar z},\mathbf z)S^{0,t} Q_n(\mathbf{\bar z},\mathbf z)^{\intercal}]_{ab}\Bigr|>\xi \Bigr)
\\
& = \sum_{\mathbf{\bar z} \in [k]^n} \sum_{ \mathbf z \in [k_0]^n} \mathbb P\Bigl(\Bigr|\frac{\tilde o_{ab}(\mathbf{\bar z},\mathbf{A}_{n\times n}^{t})}{\rho_n^t n^2} - \\
& \hspace{2cm}[Q_n(\mathbf{\bar z},\mathbf z)S^{0,t} Q_n(\mathbf{\bar z},\mathbf z)^{\intercal}]_{ab}\Bigr|>\xi| \mathbf{Z}=\mathbf z \Bigr) \mathbb P (\mathbf{Z}=\mathbf z)\\
& \leq \sum_{\mathbf{\bar z} \in [k]^n} \sum_{ \mathbf z \in [k_0]^n} 2 \exp \left(- \frac{\rho_n^t n^{2} \xi ^2}{4 S^{0,t}_{\max}  + \frac{4}{3}\xi} \right)   \mathbb P (\mathbf{Z}=\mathbf z)\\
&= 2 \exp\left(- \frac{\rho_n^t n^{2} \xi ^2}{4 S^{0,t}_{\max}  + \frac{4}{3}\xi} + n \log k \right),
\end{align*}
and this proves the first inequality of the lemma, when $a \neq b$. The first inequality in the dynamic SBM, given $\ZnT=\znT$, is obtained in the same way, using a union bound over all $\mathbf{\bar z}_n^{1:T}\in [k]^{nT}$ and integrating over $\znT$. 
\noindent Besides for $a=b$, we have that
\begin{align*}
    \tilde o_{aa}(\mathbf{\bar z},\mathbf{A}_{n\times n}^{t})&
   - \rho_n^t  n^2[Q_n(\mathbf{\bar z},\mathbf{z})S^{0,t} Q_n(\mathbf{\bar z},\mathbf{z})^{\intercal}]_{aa} \\
    &= \sum_{1\leq i<j\leq n} 2(\mathbf{A}_{ij}^{t}-  P^{0,t}_{{z}_{i}{z}_{j}}) \mathds{1}\{{\bar z}_i={\bar z}_j=a\} \\
    & \hspace{2cm} + \sum_{1 \leq i \leq n} (\mathbf{A}_{ii}^{t}-  P^{0,t}_{{z}_{i}{z}_{i}}) \mathds{1}\{{\bar z}_i=a\}.
\end{align*}
Then, for a multi-layer SBM, given $\mathbf{Z}=\mathbf{z}$, the counters $\tilde o_{aa}(\mathbf{\bar z},\mathbf{A}_{n\times n}^{t})$ correspond to the sum of $n_{aa}(\mathbf{\bar z})$ independent Bernoulli random variables multiplied by 2  \blue{where $\overline{z}_i=\overline{z}_j=a$}, given by $2 \mathbf{A}^{t}_{ij}$, with variance given by $4 P^{0,t}_{{z}_i  {z}_j} (1- P^{0,t}_{{z}_i  {z}_j} )$
, and the sum of $n_a(\mathbf{\bar z})$ independent Bernoulli random variables \blue{where $\overline{z}_i=a$}, given by $\mathbf{A}_{ii}^t$, with variance given by $ P^{0,t}_{{z}_i  {z}_i} (1- P^{0,t}_{{z}_i  {z}_i} )$.
Using Bernstein's inequality, the fact that $P^{0,t}_{{z}_i  {z}_j} (1- P^{0,t}_{{z}_i  {z}_j} ) \leq \rho_n^t S^{0,t}_{\max} $ with $S^{0,t}_{\max}= \max_{a,b}S^{0,t}_{ab}$ and the fact that $4n_{aa}(\mathbf{\bar z})+n_a(\mathbf{\bar z}) \leq 2n^2,$
 we have that for any $\delta>0$
\begin{align*}
&\mathbb P_{\pi^0, \Pvec^{0}}\Bigl(\Bigr|\tilde o_{ab}(\mathbf{\bar z},\mathbf{A}_{n\times n}^{t}) \\
&\hspace{2cm}- \rho_n^t n^2[Q_n(\mathbf{\bar z},\mathbf{z})S^{0,t} Q_n(\mathbf{\bar z},\mathbf{z})^{\intercal}]_{ab}\Bigr|>\delta \mid \mathbf{Z}=\mathbf{z}\Bigr) \\
& \leq 2 \exp \left( -\frac{ \delta ^2}{2(4n_{aa}(\mathbf{\bar z})+n_a(\mathbf{\bar z}))\rho_n^tS^{0,t}_{\max} +\frac{4}{3}\delta}  \right) \nonumber \\
& \leq 2 \exp \left( -\frac{ \delta ^2}{4n^2 \rho_n^tS^{0,t}_{\max} +\frac{4}{3}\delta}  \right).
\end{align*}
So we have the same expression as in the case where $a\neq b$. Moreover, for the dynamic SBM, using the same ideas, \blue{and recalling that, conditionally on the latent variables, the different layers $t$ are independent,} we have that for any $\delta > 0$
\begin{align*}
&\mathbb P_{\Pi^0, \Pvec^{0}}\Bigl(\blue{\Bigl|\frac{ \sum_{t=1}^T o_{aa}(\mathbf{\bar z}_n^{t},A_{n\times n}^t)}{ \rho_n n^2} }\\
 & \hspace{1.2cm}\blue{- \frac{1}{2}\sum_{t=1}^T[Q_n(\mathbf{\bar z}_n^{t},\mathbf{z}_n^t)  S^{0,t} Q_n(\mathbf{\bar z}_n^{t},\mathbf{z}_n^t)^{\intercal}]_{aa}\Bigr| }>\delta \mid  \ZnT=\znT\Bigr)\nonumber \\  
& \leq 2 \exp \left( -\frac{ \delta ^2}{2\sum_{t=1}^T(4n_{aa}(\mathbf{\bar z}^t_n)+n_a(\mathbf{\bar z}^t_n))\rho_n S^{0,t}_{\max} +\frac{4}{3}\delta}  \right) \\
&  \leq 2 \exp \left( -\frac{ \delta ^2}{4Tn^2\rho_n S^{0}_{\max} +\frac{4}{3}\delta}  \right) 
\end{align*}
with $S^{0}_{\max} =\max_t S^{0,t}_{\max} $.We can rewrite this bound with $\xi = \delta / (\rho_n n^{2})$, and using a union bound over all $\mathbf{\bar z}_n^{1:T}\in [k]^{nT}$ and integrating over $\znT \in [k_0]^{nT}$ we obtain that 
\begin{align*}
 &\mathbb P_{\Pi^0, \Pvec^{0}}\Bigl(\sup_{\mathbf{\bar z}_n^{1:T}\in [k]^{nT}} \Bigl|\frac{ \sum_{t=1}^T o_{aa}(\mathbf{\bar z}_n^{t},A_{n\times n}^t)}{ \rho_n n^2} \\
 & \hspace{2cm}- \frac{1}{2}\sum_{t=1}^T[Q_n(\mathbf{\bar z}_n^{t},\mathbf{z}_n^t)  S^{0,t} Q_n(\mathbf{\bar z}_n^{t},\mathbf{z}_n^t)^{\intercal}]_{aa}\Bigr|>\xi \Bigr) \\
 &\leq  2 \exp\left(\frac{- \rho_n n^2 \xi ^2 }{4TS^{0}_{\max}  + \frac{4}{3}\xi}+ Tn \log k\right),
\end{align*}
which proves the lemma.

\subsection{Proof of Lemma \ref{lemme_prel_2}} \label{preuve_lemme_prel_2}

Observe that 
\begin{align}\label{eqlemme_2_sparse}
 &\sum_{1 \leq a,b \leq k} \frac{n_{a}(\mathbf{z}_{n,k}^{\star}) n_{b}(\mathbf{z}_{n,k}^{\star})}{ n^{2}}  \tau \left(\frac{\tilde{o}_{a b}(\mathbf{z}_{n,k}^{\star}, A_{n \times n}^{t})}{\rho_n^t n_{a}(\mathbf{z}_{n,k}^{\star}) n_{b}(\mathbf{z}_{n,k}^{\star})}\right) \nonumber \\
 & = \sum_{1 \leq a,b \leq k} [Q_{n}(\mathbf{z}_{n,k}^{\star},\mathbf Z){\bf{1}}_{k_0}]_{a}[Q_{n}(\mathbf{z}_{n,k}^{\star},\mathbf Z){\bf{1}}_{k_0}]_{b} \times \nonumber \\
 & \hspace{2cm} \tau \left(\frac{\tilde{o}_{a b}(\mathbf{z}_{n,k}^{\star},A_{n \times n}^{t})/(\rho_n^t n^{2})}{[Q_{n}(\mathbf{z}_{n,k}^{\star},\mathbf Z){\bf{1}}_{k_0}]_{a}[Q_{n}(\mathbf{z}_{n,k}^{\star},\mathbf Z){\bf{1}}_{k_0}]_{b}}\right).
\end{align}
Then by Lemma \ref{lemme_prel_1}, for some $\epsilon_n$ such that $\epsilon_n \underset{n\rightarrow \infty}{\rightarrow} 0$ we have that
\begin{equation*}
\Bigl| \frac{\tilde{o}_{a b}(\mathbf{z}_{n,k}^{\star}, A_{n \times n}^{t})}{\rho_n^t n^{2}}  - [Q_n(\mathbf{z}_{n,k}^{\star},\mathbf Z) S^{0,t} Q_n(\mathbf{z}_{n,k}^{\star},\mathbf Z)^{\intercal}]_{ab} \Bigr|  \leq \epsilon_n,
\end{equation*}
eventually almost surely as $n\to\infty$. As $\tau$ is absolutely continuous on the compact interval $[p_{min},p_{max}]$, where $p_{min}>0$, substituting in the right-hand side of \eqref{eqlemme_2_sparse}  $\tilde{o}_{a b}(\mathbf{z}_{n,k}^{\star}, A_{n \times n}^{t})/(\rho_n^t n^{2})$ by $[Q_n(\mathbf{z}_{n,k}^{\star},\mathbf Z) S^{0,t} Q_n(\mathbf{z}_{n,k}^{\star},\mathbf Z)^{\intercal}]_{ab}$ we obtain that 
 \begin{align*}
&\sum_{1 \leq a,b \leq k} \frac{n_{a}(\mathbf{z}_{n,k}^{\star}) n_{b}(\mathbf{z}_{n,k}^{\star})}{ n^{2}}  \tau \left(\frac{\tilde{o}_{a b}(\mathbf{z}_{n,k}^{\star}, A_{n \times n}^{t})}{\rho_n^t n_{a}(\mathbf{z}_{n,k}^{\star}) n_{b}(\mathbf{z}_{n,k}^{\star})}\right) \\ &\leq \sup_{\substack {Q_n : \|Q_n\|_1=1\\Q_n^{\intercal}{\bf{1}}_{k} = n(\mathbf Z)/n}} \sum_{1\leq a, b\leq k} [Q_{n}{\bf{1}}_{k_0}]_{a}[Q_{n}{\bf{1}}_{k_0}]_{b}\times \\
& \hspace{4cm} \tau \biggl(\frac{ [Q_n S^{0,t} Q_n^{\intercal}]_{ab} }{[Q_{n}{\bf{1}}_{k_0}]_{a}[Q_{n}{\bf{1}}_{k_0}]_{b} }\biggr) + \eta_n,
\end{align*}
for some sequence $\eta_n\to 0$ as $n\to\infty$. Then taking $\lim\sup$ on both sides, and using Lemma \ref{lemme_technique} we must have that 
\begin{align*}
&\limsup_{n\to\infty} \frac{1}{2} \sum_{1 \leq a,b \leq k} \frac{n_{a}(\mathbf{z}_{n,k}^{\star}) n_{b}(\mathbf{z}_{n,k}^{\star})}{ n^{2}}  \tau  \left(\frac{\tilde{o}_{a b}(\mathbf{z}_{n,k}^{\star}, A_{n \times n}^{t})}{\rho_n^t n_{a}(\mathbf{z}_{n,k}^{\star}) n_{b}(\mathbf{z}_{n,k}^{\star})}\right)
\\ & \leq \sup_{\substack {R : \| R\|_1=1\\R^{\intercal}{\bf{1}}_{k} = \pi}} \frac{1}{2} \sum_{1\leq a, b\leq k} [R{\bf{1}}_{k_0}]_{a}[R{\bf{1}}_{k_0}]_{b} \tau \biggl(\frac{ [R S^{0,t} R^{\intercal}]_{ab} }{[R{\bf{1}}_{k_0}]_{a}[R{\bf{1}}_{k_0}]_{b} }\biggr),
\end{align*}
almost surely. 
The supremum  in the right-hand side is a maximum of a convex function over a convex polyhedron defined by $\{R\colon \|R\|_1=1, R^\intercal{\bf {1}}_k = \pi \}$. Then, the maximum must be attained at one of the vertices of the polyhedron;   that is, on those matrices $R$ such that at most one entry per column is greater than zero. Since $\pi_a>0$ for all $a\in \{1,\dots,k_0\}$, it follows that each column must have at least one strictly positive entry. Thus, the maximum is achieved on matrices where one and only one entry by column is greater than zero. We denote by $R^{\star}$ one of these maxima (if there is more than one) and let 
\begin{align*} 
    \pi^\star_a &=[R^{\star}{\bf{1}}_{k_0}]_a, \quad a\in \{1,\dots,k\} \nonumber\\
    S^{\star,t}_{ab} &= \frac{[R^{\star}S^{0,t}(R^{\star})^{\intercal}]_{ab}}{[R^{\star}{\bf{1}}_{k_0}{\bf{1}}_{k_0}^{\intercal} (R^{\star})^{\intercal}]_{ab}},\quad a,b\in \{1,\dots,k\}.
\end{align*}
Then 
\begin{align*}
\sup_{\substack {R\colon \| R\|_1=1\\R^{\intercal}{\bf{1}}_{k} = \pi}} \frac{1}{2} \sum_{1\leq a, b\leq k} [R{\bf{1}}_{k_0}]_{a} & [R{\bf{1}}_{k_0}]_{b} \tau \biggl(\frac{ [R S^{0,t} R^{\intercal}]_{ab} }{[R{\bf{1}}_{k_0}]_{a}[R{\bf{1}}_{k_0}]_{b} }\biggr)\\
&= \frac{1}{2} \sum_{1 \leq a,b \leq k} \pi_{a}^{\star} \pi_{b}^{\star} \tau \left(S^{\star,t}_{ab}\right).
\end{align*}
This concludes the proof of the lemma.

\subsection{Proof of Lemma \ref{lemme_prel_3}} \label{preuve_lemme_prel_3}

As $R^{\star}$ has one and only one non-zero entry in each column, we have that there is a surjective function $h : [k_0]\to[k]$ connecting each community in $[k_0]$ (columns of $R^{\star}$) with its corresponding community in $[k]$ (line with non-zero entry). Then for $k=k_0-1$,  there are $k-1$ communities in 
$\{1,\dots,k_0\}$ that are mapped into $k-1$ communities in  $\{1,\dots,k\}$ and two communities in $\{1,\dots,k_0\}$ that are mapped into  a single community in   $\{1,\dots,k\}$. 
Without loss of generality, assume that the communities $k_0-1$ and $k_0$ satisfy $h(k_0-1)=h(k_0)=k=k_0-1$, and for all $a\leq k_0-1$ we have that $h(a)=a$. Moreover, as $(R^{\star})^{\intercal}{\bf 1}_{k}=\pi$, we must have that  the non-zero entries are given by 
\begin{align*}
R_{aa}^{\star}&=\pi_a, \quad 1\leq a\leq k_0-1\\
R_{k_0-1,k_0}^{\star}&= \pi_{k_0}.
\end{align*}
Then, the parameters $\pi^\star$ and $S^{\star,t}$ defined in \eqref{def_S_star_pi_star} are given by 
\begin{align*}
\pi_{a}^{\star}&=\pi_a, \quad 1\leq a\leq k_0-1 \\
\pi_{k_{0}-1}^{\star}&=\pi_{k_0-1}+\pi_{k_0},
\end{align*}
and 
\begin{align*}
S_{ab}^{\star,t}&= S^{0,t}_{ab}, \quad 1\leq a,b\leq k_0-2\\
S_{a,k_0-1}^{\star,t}&=\frac{\pi_{k_0-1}S^{0,t}_{a,k_0-1}+\pi_{k_0}S^{0,t}_{a,k_0}}{\pi_{k_0-1}+\pi_{k_0}},\quad1\leq l\leq k_0-2,\\
S_{k_0-1,k_0-1}^{\star,t}& \\
&\hspace{-1cm}=\frac{\pi_{k_0-1}^2S^{0,t}_{k_0-1,k_0-1}+2\pi_{k_0-1}\pi_{k_0}S^{0,t}_{k_0-1,k_0}+\pi_{k_0}^2S^{0,t}_{k_0,k_0}}{\pi_{k_0-1}^2+2\pi_{k_0-1}\pi_{k_0}+\pi_{k_0}^2}.
\end{align*}
Observe that for all  $1\leq a\leq k_0-2$ we have that $[S^{\star,t}\pi^\star]_a=[S^{0,t}\pi]_a$ then for all $1\leq a,b\leq k_0-2$
\begin{equation*}
\pi_{a}^{\star} \pi_{b}^{\star} \tau \left(S^{\star,t}_{ab}\right) =  \pi_{a} \pi_{b} \tau \left(S^{0,t}_{ab}\right).
\end{equation*}
On the other hand we have that for $1\leq a\leq k_0-2$ it follows, by using twice the log-sum inequality, that
\begin{align*}
\pi_a^\star\pi_{k_0-1}^\star\tau &\left(S^{\star,t}_{a,k_0-1}\right) \\ 
&\leq  \pi_a \pi_{k_0-1} \tau \left(S^{0,t}_{a,k_0-1}  \right) +  \pi_a \pi_{k_0} \tau \left(S^{0,t}_{a,k_0} \right).
\end{align*}
Moreover, we have that the inequality must be strict unless
\begin{equation} \label{condition_1_s}
S^{0,t}_{a,k_0-1}=S^{0,t}_{a,k_0}, \qquad \text{for all }a\leq k_0-2\,.
\end{equation}
On the other hand, for $a= k_0-1$ and $b=k_0-1$, also by using twice the log-sum inequality  we have that 
\begin{align*}
&\pi_{k_0-1}^\star\pi_{k_0-1}^\star \tau \left(S^{\star,t}_{k_0-1,k_0-1}\right)\\
& \leq  \pi_{k_0-1}^2 \tau \left( S^{0,t}_{k_0-1,k_0-1} \right) +2\pi_{k_0-1}\pi_{k_0} \tau \left(S^{0,t}_{k_0-1,k_0}  \right) \\
& \hspace{3.6cm}+\pi_{k_0}^2 \tau \left(S^{0,t}_{k_0,k_0} \right),
\end{align*}
with equality if and only if
\begin{equation} \label{condition_2_s}
S^{0,t}_{k_0-1,k_0-1}= S^{0,t}_{k_0-1,k_0}=S^{0,t}_{k_0,k_0}.
\end{equation}
From \eqref{condition_1_s} and \eqref{condition_2_s} we obtain that the inequality in Lemma \ref{lemme_prel_3} must be strict unless 
\begin{equation*}
S^{0,t}_{a,k_0-1} = S^{0,t}_{a,k_0} \qquad \text{for all }a\leq k_0\,.
\end{equation*}

\subsection{Proof of Lemma \ref{lemme_prel_2_dyn}} \label{preuve_lemme_prel_2_dyn}

Observe that 
\begin{align}\label{eqlemme_2_sparse_dyn}
&\sum_{1 \leq a \neq b \leq k}  \frac{n_{ab}(\mathbf{z}_{n,k}^{\star,t}) }{ n^{2}} \tau \left( \frac{{o}_{a b}(\mathbf{z}_{n,k}^{\star,t}, A_{n\times n}^{t})}{\rho_n n_{ab}(\mathbf{z}_{n,k}^{\star,t})}\right) \nonumber\\ 
 & =
 \sum_{1 \leq a \neq b \leq k} [Q_{n}(\mathbf{z}_{n,k}^{\star,t},\mathbf{Z}_{n}^{t}){\bf{1}}_{k_0}]_{a}[Q_{n}(\mathbf{z}_{n,k}^{\star,t},\mathbf{Z}_{n}^{t}){\bf{1}}_{k_0}]_{b} \times \nonumber \\
 & \hspace{2cm} \tau \left(\frac{{o}_{a b}(\mathbf{z}_{n,k}^{\star,t}, A_{n\times n}^t)/(\rho_n n^{2})}{[Q_{n}(\mathbf{z}_{n,k}^{\star,t},\mathbf{Z}_{n}^{t}){\bf{1}}_{k_0}]_{a}[Q_{n}(\mathbf{z}_{n,k}^{\star,t},\mathbf{Z}_{n}^{t}){\bf{1}}_{k_0}]_{b} }\right).
\end{align}
Then by Lemma \ref{lemme_prel_1}, for some $\epsilon_n$ such that $\epsilon_n \underset{n\rightarrow \infty}{\rightarrow} 0$ we have that
\begin{equation*}
\Bigl| \frac{{o}_{a b}(\mathbf{z}_{n,k}^{\star,t}, A_{n\times n}^t)}{\rho_n n^2} - [Q_{n}(\mathbf{z}_{n,k}^{\star,t},\mathbf{Z}_{n}^{t})S^{0,t} Q_{n}(\mathbf{z}_{n,k}^{\star,t},\mathbf{Z}_{n}^{t})^{\intercal}]_{ab}\Bigr|  \leq \epsilon_n
\end{equation*}
eventually almost surely as $n\to\infty$. As $\tau$ is absolutely continuous on the compact interval $[p_{min},p_{max}]$, where $p_{min}>0$, substituting in the right-hand side of \eqref{eqlemme_2_sparse_dyn} ${o}_{a b}(\mathbf{z}_{n,k}^{\star,t}, A_{n\times n}^t)/(\rho_n n^{2})$ by $ [Q_{n}(\mathbf{z}_{n,k}^{\star,t},\mathbf{Z}_{n}^{t})S^{0,t} Q_{n}(\mathbf{z}_{n,k}^{\star,t},\mathbf{Z}_{n}^{t})^{\intercal}]_{ab}$  we obtain that 
 \begin{align*}
&\sum_{1 \leq a \neq b \leq k}  \frac{n_{ab}(\mathbf{z}_{n,k}^{\star,t}) }{ n^{2}}  \tau \left( \frac{{o}_{a b}(\mathbf{z}_{n,k}^{\star,t}, A_{n\times n}^{t})}{\rho_n n_{ab}(\mathbf{z}_{n,k}^{\star,t})}\right)\\ 
& \leq \sup_{\substack {Q_n : \|Q_n\|_1=1\\Q_n^{\intercal}{\bf{1}}_{k} = n(\mathbf{Z}_n^t)/n}} \sum_{1\leq a \neq b\leq k} [Q_{n}{\bf{1}}_{k_0}]_{a}[Q_{n}{\bf{1}}_{k_0}]_{b}  \times \\
&\hspace{4cm}\tau \biggl(\frac{[Q_n S^{0,t} Q_n^{\intercal}]_{ab}}{[Q_{n}{\bf{1}}_{k_0}]_{a}[Q_{n}{\bf{1}}_{k_0}]_{b}  }\biggr) + \eta_n,
\end{align*}
for some sequence $\eta_n\to 0$ as $n\to\infty$. Then taking $\lim\sup$ on both sides, and using Lemma \ref{lemme_technique} we must have that 
\begin{align*}
& \limsup_{n\to\infty} \frac{1}{2}\sum_{1 \leq a \neq b \leq k}  \frac{n_{ab}(\mathbf{z}_{n,k}^{\star,t}) }{ n^{2}} \tau \left( \frac{{o}_{a b}(\mathbf{z}_{n,k}^{\star,t}, A_{n\times n}^{t})}{\rho_n n_{ab}(\mathbf{z}_{n,k}^{\star,t})}\right) \\ 
& \leq \sup_{\substack {R: \| R\|_1=1\\R^{\intercal}{\bf{1}}_{k} = \alpha}} \frac{1}{2} \sum_{1\leq a\neq b\leq k} [R{\bf{1}}_{k_0}]_{a}[R{\bf{1}}_{k_0}]_{b} \tau \biggl(\frac{ [R S^{0,t} R^{\intercal}]_{ab} }{[R{\bf{1}}_{k_0}]_{a}[R{\bf{1}}_{k_0}]_{b} }\biggr),
\end{align*}
almost surely. 
The supremum  in the right-hand side is a maximum of a convex function over a convex polyhedron defined by $\{R\colon \|R\|_1=1, R^T{\bf {1}}_k = \alpha \}$. Then, the maximum must be attained at one of the vertices of the polyhedron;   that is, on those matrices $R$ such that at most one entry per column is greater than zero. Since $\alpha_a>0$ for all $a\in \{1,\dots,k_0\}$, it follows that each column must have at least one strictly positive entry. Thus, the maximum is achieved on matrices where one and only one entry by column is greater than zero. We denote by $R^{\star}$ one of these maxima (if there is more than one) and let 
\begin{align*} 
    \alpha^\star_a &=[R^{\star}{\bf{1}}_{k_0}]_a, \quad a\in \{1,\dots,k\} \nonumber\\
    S^{\star,t}_{ab} &= \frac{[R^{\star}S^{0,t}(R^{\star})^{\intercal}]_{ab}}{[R^{\star}{\bf{1}}_{k_0}{\bf{1}}_{k_0}^{\intercal} (R^{\star})^{\intercal}]_{ab}},\quad a \neq b\in \{1,\dots,k\}.
 \end{align*}
Then 
\begin{align*}
 \sup_{\substack {R: \| R\|_1=1\\R^{\intercal}{\bf{1}}_{k} = \alpha}} \frac{1}{2} \sum_{1\leq a\neq b\leq k} [R{\bf{1}}_{k_0}]_{a} &[R{\bf{1}}_{k_0}]_{b} \tau \biggl(\frac{ [R S^{0,t} R^{\intercal}]_{ab} }{[R{\bf{1}}_{k_0}]_{a}[R{\bf{1}}_{k_0}]_{b} }\biggr) \\
& = \frac{1}{2} \sum_{1 \leq a \neq b \leq k} \alpha_{a}^{\star} \alpha_{b}^{\star} \tau \left(S^{\star,t}_{ab}\right).
\end{align*}
Moreover, we have that 
\begin{align*}
 &\sum_{1 \leq a \leq k} \frac{\sum_{t=1}^T n_{aa}(\mathbf{z}_{n,k}^{\star,t})}{n^2}   \tau \left( \frac{\sum_{t=1}^To_{aa}(\mathbf{z}_{n,k}^{\star,t},A_{n\times n}^t)}{\sum_{t=1}^T\rho_n n_{aa}(\mathbf{z}_{n,k}^{\star,t})}\right) \\
 & = \sum_{1 \leq a\leq k}  \frac{1}{2}\sum_{t=1}^T[Q_{n}(\mathbf{z}_{n}^{\star,t},\mathbf{Z}_n^t){\bf{1}}_{k_0}]_{a}^2 \tau \left(\frac{2 {o}_{aa}(\mathbf{z}_{n}^{\star,t}, A_{n \times n}^t)/(\rho_n n^{2})}{\sum_{t=1}^T[Q_{n}(\mathbf{z}_{n}^{\star,t},\mathbf{Z}_n^t){\bf{1}}_{k_0}]_{a}^2 }\right).
\end{align*}
Then with the same method as before, and using Lemma \ref{lemme_prel_1} and Lemma \ref{lemme_technique} we have that
\begin{align*}
\limsup_{n\to\infty}& \sum_{1 \leq a \leq k} \frac{\sum_{t=1}^T n_{aa}(\mathbf{z}_{n,k}^{\star,t})}{n^2}  \tau \left( \frac{\sum_{t=1}^To_{aa}(\mathbf{z}_{n,k}^{\star,t},A_{n\times n}^t)}{\sum_{t=1}^T\rho_n n_{aa}(\mathbf{z}_{n,k}^{\star,t})}\right)\\
& \leq \sup_{\substack {R: \| R\|_1=1\\R^{\intercal}{\bf{1}}_{k} = \alpha}} \frac{T}{2} \sum_{1\leq a\leq k}  [R{\bf{1}}_{k_0}]_{a}^2  \tau \biggl(\frac{ \sum_{t=1}^T[R S^{0,t} R^{\intercal}]_{aa} }{T[R{\bf{1}}_{k_0}]_{a}^2 }\biggr),
\end{align*}
almost surely. 
Then, as before, the supremum  in the right-hand side is a maximum of a convex function over a convex polyhedron. Then, the maximum must be attained at one of the vertices of the polyhedron. We denote by $R^{\star}$ one of these maxima (if there is more than one) and let 
\begin{align*} 
    \alpha^\star_a &=[R^{\star}{\bf{1}}_{k_0}]_a, \quad a\in \{1,\dots,k\} \\
    S^{\star}_{aa} &= \frac{\sum_{t=1}^T[R^{\star}S^{0,t}(R^{\star})^{\intercal}]_{aa}}{T[R^{\star}{\bf{1}}_{k_0}]_{a}^2},\quad a \in \{1,\dots,k\}.
 \end{align*}
This concludes the proof of the lemma.

\subsection{Proof of Lemma \ref{lemme_prel_3_dyn}} \label{preuve_lemme_prel_3_dyn}
As $R^{\star}$ has one and only one non-zero entry in each column, we have that there is a surjective function $h_t : [k_0]\to[k]$ connecting each community in $[k_0]$ (columns of $R^{\star}$) with its corresponding community in $[k]$ (line with non-zero entry). Then for $k=k_0-1$,  there are $k-1$ communities in 
$\{1,\dots,k_0\}$ that are mapped into $k-1$ communities in  $\{1,\dots,k\}$ and two communities ($u(t),v(t)$) in $\{1,\dots,k_0\}$ that are mapped into a single community ($u(t)$) in $\{1,\dots,k\}$. The communities $u(t)$ and $v(t)$ satisfy $h_t(u(t))=h_t(v(t))=u(t)=v(t)$, and for all $a \neq u(t),v(t)$ we have that $h_t(a)=a$. Moreover, as $(R^{\star})^{\intercal}{\bf 1}_{k}=\alpha$, we must have that  the non-zero entries are given by 
\begin{align*}
R_{aa}^{\star}&=\alpha_a, \quad a \neq v(t)\\
R_{u(t),v(t)}^{\star}&= \alpha_{v(t)}.
\end{align*}
Then, the parameters $\alpha^\star$ and $S^{\star,t}$ defined in \eqref{def_S_star_alpha_star} are given by 
\begin{align*}
\alpha_{a}^{\star}&=\alpha_a, \quad  a \neq u(t) \\
\alpha_{u(t)}^{\star}&=\alpha_{u(t)}+\alpha_{v(t)},
\end{align*}
and
\begin{align*} 
S_{ab}^{\star,t}&= S^{0,t}_{ab}, \quad 1\leq a \neq b\leq k_0, \text{ and } a,b \neq u(t),v(t)\\
S_{a,u(t)}^{\star,t}&=\frac{\alpha_{u(t)}S^{0,t}_{a,u(t)}+\alpha_{v(t)}S^{0,t}_{a,v(t)}}{\alpha_{u(t)}+\alpha_{v(t)}},\quad a\neq u(t),\\
S_{aa}^{\star}=&\\
&\hspace{-2cm}\frac{\sum_{t=1}^T \left[ \alpha_{a}^2S^{0}_{aa}+(2\alpha_{u(t)}\alpha_{v(t)}S^{0,t}_{u(t),v(t)}+\alpha_{v(t)}^2S^{0}_{v(t),v(t)})\mathds{1}\{a=u(t)\} \right]}{\sum_{t=1}^T \left[\alpha_a + \alpha_{v(t)}\mathds{1}\{a=u(t)\} \right]^2}.
\end{align*}
Observe that for all $1\leq t \leq T$ and $a,b \neq u(t)$ such that $a\neq b$ we have that
\begin{equation*}
\alpha_{a}^{\star} \alpha_{b}^{\star} \tau \left(S^{\star,t}_{ab}\right) =  \alpha_{a} \alpha_{b} \tau \left(S^{0,t}_{ab}\right).
\end{equation*}
On the other hand we have that for $a\neq u(t)$ it follows, by using twice the log-sum inequality, that 
\begin{align*}
\alpha_a^\star\alpha_{u(t)}^\star\tau \left(S^{\star,t}_{a,u(t)}\right) \leq \alpha_a \alpha_{u(t)} \tau \left(S^{0,t}_{a,u(t)}  \right) +  \alpha_a \alpha_{v(t)} \tau \left(S^{0,t}_{a,v(t)} \right).
\end{align*}
Moreover, we have that the inequality must be strict unless
\begin{equation} \label{condition_1_s_dyn}
S^{0,t}_{a,u(t)}=S^{0,t}_{a,v(t)}, \qquad \text{for all }a\neq u(t).
\end{equation}
On the other hand, for all $a$, also by using twice the log-sum inequality we have that 
\begin{align*}
\sum_{t=1}^T(\alpha_{a}^\star)^2 \tau \left(S^{\star}_{aa}\right) &\leq T\alpha_a^2 \tau (S_{aa}^0) \\
&\qquad+ 2 \sum_{t=1}^T \alpha_{u(t)} \alpha_{v(t)} \tau(S_{u(t),v(t)}^{0,t} ) \mathds{1}_{\{a=u(t)\}} \\
&\qquad+\sum_{t=1}^T \alpha_{v(t)}^2 \tau (S_{v(t),v(t)} ^0) \mathds{1}_{\{a=u(t)\}}
\end{align*}
with equality if and only if, for all $1\leq t \leq T$,
\begin{equation} \label{condition_2_s_dyn}
S^{0}_{u(t),u(t)}= S^{0,t}_{u(t),v(t)}=S^{0}_{v(t),v(t)}.
\end{equation}
From \eqref{condition_1_s_dyn} and \eqref{condition_2_s_dyn} we obtain that the inequality in Lemma \ref{lemme_prel_3_dyn} must be strict unless 
\begin{equation*}
S^{0,t}_{a,u(t)} =  S^{0,t}_{a,v(t)} \qquad \text{for all }a\leq k_0\,
\end{equation*}
for all $1 \leq t \leq T$ which is a contradiction with the hypothesis for the matrix $S^{0,t}$. 

\section{\blue{Empirical illustration under model misspecification}}
\label{app:misspecification}

\blue{In this appendix, we provide a limited empirical illustration of the behavior of the MLSBM model selection criterion under model misspecification.}

\blue{Specifically, networks were generated according to a DynSBM data-generating process and subsequently estimated using the MLSBM criterion with the penalization defined in Equation~\eqref{pen}. The accuracy of the estimated clustering is reported as a function of the number of nodes.}

\blue{The results, displayed in Figure~\ref{fig:misspecification}, suggest that the MLSBM-based procedure remains reasonably accurate even when the underlying model is misspecified. While this experiment does not allow for a direct comparison with the DynSBM criterion---due to the lack of a full variational Bayes implementation for DynSBM---it nonetheless provides empirical evidence of a certain robustness of the approach.}

\blue{A more systematic study of model misspecification, including a full implementation of inference procedures for DynSBM, is left for future work.}

\begin{figure}[ht]
    \centering
     \includegraphics[ height=5cm]{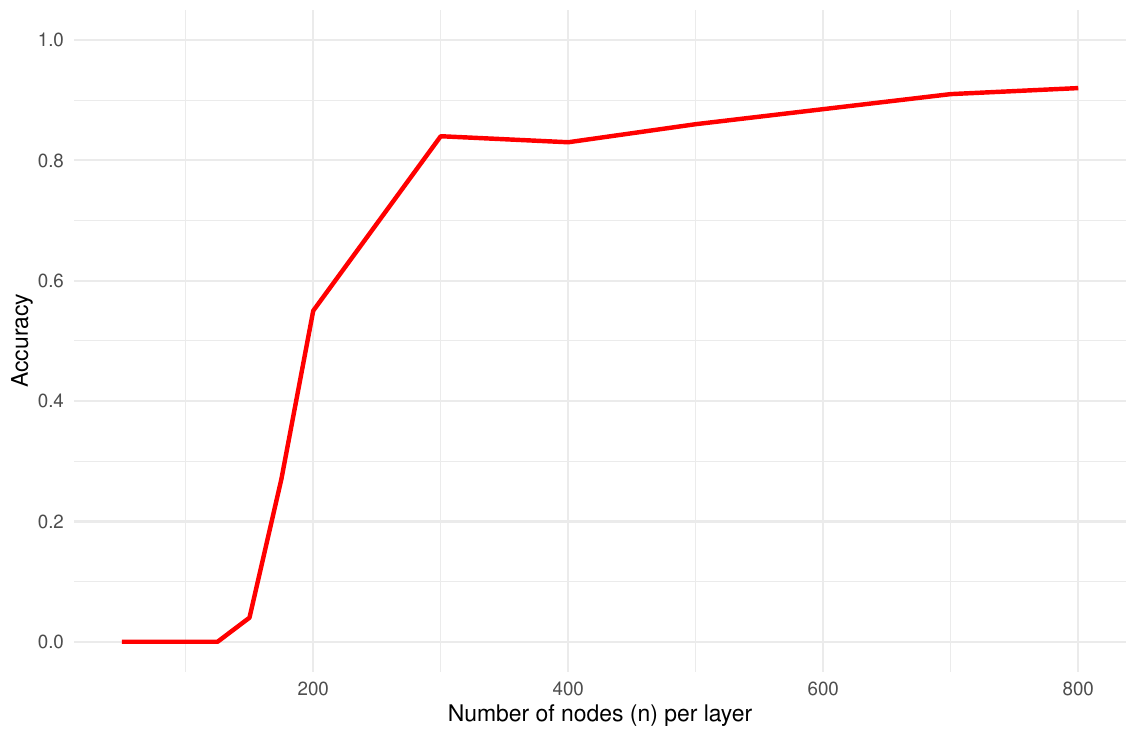}
    \caption{\blue{Clustering accuracy of the MLSBM estimator when data are generated from a DynSBM, as a function of the number of nodes.}}
    \label{fig:misspecification}
\end{figure}

\section*{Acknowledgments}
My work would not have led to the writing of this article without the kindness and wise advice of Catherine Matias, my thesis supervisor. I sincerely thank her.

\begin{IEEEbiographynophoto}{Lucie Arts}
spent two years in "classes préparatoires" before continuing her studies at Sorbonne University. She graduated with a Master's degree in Statistics and joined the LPSM in 2024, where she is currently a PhD student under the supervision of Catherine Matias. Her research focuses on estimation in SBMs.
\end{IEEEbiographynophoto}

\end{document}